\documentclass{article}
\usepackage[utf8]{inputenc}

\usepackage[margin=1in]{geometry}
\usepackage{amsfonts}
\usepackage{amsthm}
\usepackage{amsmath}
\usepackage{mathtools}
\usepackage{amssymb}
\usepackage{siunitx}
\usepackage{physics}
\usepackage{enumerate}
\usepackage{multicol}
\usepackage[noadjust]{cite}
\usepackage{caption}
\usepackage{subcaption}
\usepackage{graphicx}
\usepackage{tcolorbox}
\usepackage{dsfont}
\usepackage{xcolor}
\usepackage{tikz-cd}
\usetikzlibrary{backgrounds}
\usepackage{subfiles}
\usepackage{longtable}
\usepackage{hyperref}
\usepackage{bbm}
\newtheorem{theorem}{Theorem}
\newtheorem*{theorem*}{Theorem}

\newtheorem{conj}[theorem]{Conjecture}
\newtheorem{lemma}[theorem]{Lemma}
\newtheorem{defn}{Definition}

\newtheorem*{defn*}{Definition}
\newtheorem{propn}[theorem]{Proposition}

\newcommand{\B}[1]{\mathbb{#1}}
\newcommand{\C}[1]{\mathcal{#1}}
\newcommand{\oop}{\preccurlyeq}
\newcommand{\is}{\mathbbm{1}}

\usepackage[skip=7.5pt plus1pt, indent=20pt]{parskip}

\title{Branching Brownian motion with rank-based selection and reaction-diffusion equations}
\author{Jacob Mercer}
\date{}  

\begin{document}
\maketitle

\begin{abstract}
    We consider a family of branching-selection particle systems in which particles branch at time dependent rate $r$ and are killed with a probability which is dependent on their rank via some function $\psi$. We show that, under fairly minimal conditions, the hydrodynamic limit of such a system is given by the reaction-diffusion equation $U_t = \frac12 U_{xx} + r(t)G(U)$ with nonlinearity $G(U)$ which is a function of $\psi$. This is a significant generalisation of the well-studied $N$-BBM process, and is similar to the family of `$(b,D)$-BBM' processes described by Groisman \& Soprano-Loto in \cite{rankDependentBBM}. On the one hand, this allows us to understand common reaction-diffusion equations as limits of interacting particle systems with simple description. On the other hand, the asymptotic behaviour of solutions of the reaction-diffusion the PDEs can help us predict the asymptotic properties of the associated particle systems. We give general conditions under which the branching-selection particle system has an asymptotic velocity, and describe the velocity up to order $(\log N)^{-2}$; furthermore, we describe the connection between this velocity and the spreading speeds and travelling waves of the corresponding reaction-diffusion equation. This provides a partial weak selection principle.
\end{abstract}

\section{Introduction}

Reaction diffusion equations are a broad class of partial differential equations, which find applications in a number of areas of physical sciences, including ecology, cell biology, and flame propagation. We consider reaction-diffusion equations of the general form:
$$U_t = \frac12 U_{xx} + Y(x,t,U),$$
where $Y:\B{R}\times [0,\infty)\times \B{R}\to \B{R}$. A well-known example is the F-KPP equation, which was first studied in the 1930s independently by Fisher \cite{fisher} and Kolmogorov, Petrovskii, and Piscunov \cite{kpp} in the context of population genetics. The F-KPP equation is an equation of the form $U_t=\frac12U_{xx}+G(U)$ with $G(0)=G(1)=0$, $G(x)>0$ for $x\in (0,1)$, and $G'(x)\leq G'(0)$ for $x\in [0,1]$; a typical example being the Fisher equation $U_t=\frac12U_{xx}+U(1-U)$. An important property of the F-KPP equation is the existence of travelling wave solutions and travelling fronts; it is well known that the F-KPP equation $U_t=\frac12U_{xx}+G(U)$ has a travelling wave solution for every speed $c\geq \sqrt{2G'(0)}$. Thus the Fisher equation $U_t=\frac12U_{xx}+U(1-U)$ has travelling wave solution for all speeds $c\geq \sqrt{2}$. The situation is more complex in the general case, when $G$ is not positive everywhere on $[0,1]$. 

The PDE $U_t = \frac12 U_{xx}+ U(1-U)$ can also be understood as the hydrodynamic limit of an interacting particle system \cite{fkppScalingLimit} which can be described as follows: consider $N$ particles moving on $\B{R}$ as independent Brownian motions. At rate $N$, draw uniformly at random two particles from the population and move the leftmost of the two particles to the location of the rightmost particle.

Groisman and Soprano-Loto later generalised this when they describe their so-called `$(b,D)$-BBM' process. In the $(b,D)$-BBM, particles behave as independent Brownian motions on $\B{R}$, with branching described by the function $b:[0,1]\to[0,\infty)$ and deletion by the increasing function $D:[0,1]\to[0,1]$. More precisely, the $j$\textsuperscript{th} leftmost particle branches at rate $b\big(\frac{j-1}{N_t-1}\big)$, where $N_t$ is the number of particles in the system at time $t$. Simultaneously to the $j$ \textsuperscript{th} particle branching, the $i$\textsuperscript{th} leftmost particle, where $i< j$, is killed with probability $D\big(\frac{i}{j-1}-\big)-D\big(\frac{i-1}{j-1}-\big)$ and this includes the possibility that no particle is deleted. The family of processes they describe is relatively general, and in particular includes branching Brownian motion (BBM), and the fixed population size $N$-BBM process (see \cite{toughBeres},\cite{hydroNBBM},\cite{maillard}), among many others. In their work, they conjecture that the particle system has the hydrodynamic limit:
$$u_t(x,t) = \frac12 u_{xx}(x,t) + u(x,t)\left(b(U(x,t))-\int_{U(x,t)}^1 b(r)\frac1r D(dr)\right),$$
where $u$ is the limiting empirical density and $U$ the corresponding cumulative distribution. 

Another generalisation of the $N$-BBM is studied by Atar \cite{atar}, who considers an \textit{injection-branching-selection} system, in which the branching and selection steps are decoupled. The system proposed in \cite{atar} starts with $N$ Brownian motions on $\B{R}$, each branching at rate $\kappa\geq 0$, and subsequently particles are added to the system (\textit{injected}) according to a random point process $\alpha^N(dx,dt)$ and leftmost particles are deleted according to a random  process $J^N$. As with the $(b,D)$-BBM, this setup allows for a variable number of particles. Subsequently Atar describes a weak formulation of the corresponding hydrodynamic limit.

In this paper we study a branching-selection particle system with rank-dependent selection which can be described as follows. $N$ particles move as independent Brownian motions on $\B{R}$, and each particle branches into two particles independently at time dependent rate $r(t)$. Simultaneously with each branching event, we kill a particle, killing the $i$\textsuperscript{th} leftmost particle with probability $\int_{(i-1)/N}^{i/N}\psi(s)ds$, where $\psi$ is a positive, bounded, and continuous function with $\int_0^1 \psi(s)ds=1$. We will call this process the $(\psi,r,N)$-BBM. Note that unlike the $(b,D)$-BBM, this allows for the rank of the deleted particle to be greater than the rank of the branching particle. Then we will show that the empirical cumulative distribution function of the $(\psi,r,N)$-BBM process converges to the unique solution to the following reaction-diffusion PDE:
\begin{equation} \label{PDEinIntro} U_t(x,t)=\frac12U_{xx}(x,t)+r(t)\left(U(x,t)-\int_{1-U(x,t)}^1 \psi(s)ds\right).\end{equation}

The consequences of this result are two-fold. Firstly, any PDE $U_t = \frac12 U_{xx} + r(t)G(U)$, where $G(0)=G(1)=0$ and $G'(x)\leq 1$ for $x\in [0,1]$ can be understood as the hydrodynamic limit of a $(\psi,r,N)$-BBM with branching rate $r$ and selection function $\psi(x)=1-G'(1-x)$. Secondly, we may consider functions $G$ for which the asymptotic behaviour of the equation $U_t = \frac12 U_{xx} + G(U)$ is known, and discover what this implies about the asymptotic behaviour of the corresponding interacting particle system. Examples are given in Section \ref{travellingSection}. 

Another question about the $(\psi,r,N)$-BBM which naturally arises is that of its asymptotic velocity. By coupling the process to versions of the $N$-BBM process, we prove that, when $r(t)\equiv 1$ and under certain conditions on $\psi$, its asymptotic velocity, $v^\psi_N$, matches that of the $N$-BBM up to order $O((\log N)^{-2})$. That is,
$$v^\psi_N = \sqrt{2} - \frac{\pi^2}{\sqrt{2}(\log N)^2} + o\left(\frac{1}{(\log N)^2}\right).$$

We may notice that when $r(t)\equiv 1$, $\psi(1)=0$, and and $\int_0^x \psi(s)ds \leq x$ for all $x\in [0,1]$, then $\lim_{N\to\infty}v^\psi_N=\sqrt{2}=\sqrt{2G'(0)}$ is the minimal travelling wave speed for the PDE (\ref{PDEinIntro}). This is an example of a \textit{weak selection principle}; the $N$ particle system has an asymptotic velocity which `selects' the minimal travelling wave speed of the limiting PDE as $N\to\infty$. See for example \cite{toughBeres}, \cite{hydroNBBM}, \cite{sameSelec}. Thus we aim to the answer the question of when the following diagram holds:
\begin{center}\begin{tikzcd}[column sep=6em, row sep=4em] (\psi,r,N)\text{-BBM Process} \arrow[r, "\text{Hydrodynamic limit}\atop N \to \infty"] \arrow[d, "{\text{Asymptotic velocity}\atop t \to \infty}"']  & \text{Solution to PDE (\ref{PDEinIntro})} \arrow[d, dashed, no head, "\text{Minimal wave speed}"] \\ v_N \arrow[r, "N \to \infty"] & \sqrt{2} \end{tikzcd}\end{center}
A discussion of this appears in Section \ref{travellingSection}.

\section{Construction of the process} \label{constructionSection}

Before stating the main results of this paper, we will formally construct the $(\psi,r,N)$-BBM process as a measure-valued random process $(\mu^N_t)_{t\geq 0}$ constructed from i.i.d Brownian motions, a Poisson process, and discrete random variables. We will write $\mu_t^N:=\frac1N \sum_{i=1}^N \delta_{X_i^N(t)}$ where $X^N(t)=(X_1^N(t),X_2^N(t)\ldots,X_N^N(t))$ is the vector of locations of the $N$ particles at time $t$. Define the function $\Theta^N:\B{R}^N\to\B{R}^N$ to be the function which sorts an $N$ dimensional vector into increasing order. Furthermore, let $\Theta^N_j(\underline{x})$ denote the $j$\textsuperscript{th} element of the vector $\Theta^N(\underline{x})$, which is the $j$\textsuperscript{th }smallest element of the vector $\underline{x}$.

Let the initial condition $\mu^N_0=\rho^N$ where $\rho^N$ is an atomic measure, the sum of $N$ atoms of weight $1/N$, and suppose that $\rho^N$ converges weakly to the probability measure $\rho$ which is absolutely continuous with respect to the Lebesgue measure.

Now let $(\C{N}(t))_{t\geq 0}$ be a Poisson process with time-dependent intensity $Nr(t)$, where $r$ is continuous and bounded. Let $t_1,t_2,\ldots$ be the discontinuities of the process, with $t_0=0$. At each time $t_m$ we will order the elements of $X^N(t)$ so that $X_1^N(t_m)\leq X_2^N(t_m)\leq  \ldots \leq X_N^N(t_m)$, but we may have $X_i(t)>X_j(t)$ for $i<j$ and $t\notin\{t_0,t_1,t_2,\ldots\}$. Let $\C{I}=(I_n)_{n=1,2,\ldots}$ be a sequence of i.i.d. uniform random variables on $\{1,2,\ldots,N\}$, and let $\C{J}=(J_n)_{n=1,2,\ldots}$ be a sequence of i.i.d. random variables with $\B{P}(J_n=k)=\int_{(k-1)/N}^{k/N}\psi(s)ds$ for $k=1,2,\ldots,N$, where $\psi:[0,1]\to[0,\infty)$ is assumed to be positive and continuous with $\int_0^1 \psi(s)ds=1$.

%Now let \JB{$\{(\C{N}_{i,j}(t))_{t\geq 0}, 1\le i, j \le N\}$} be a family of $N^2$ Poisson processes, where $\C{N}_{i,j}$ has \JB{measure-}intensity \JB{$r(t)\left[\int_{(j-1)/N}^{j/N}\psi(r)dr\right]dt$}. Let $t_1,t_2,\ldots$ be the times of discontinuity of a Poisson process \JB{which one?}, and let $(i_1,j_1),(i_2,j_2),\ldots$ be the corresponding indices\JB{too vague, reformulate}. That is, $t_0=0$ and inductively $t_k:=\inf\{t>t_{k-1}: \C{N}_{i,j}(t)\neq \C{N}_{i,j}(t-)\text{ for some }(i,j)\}$, and $(i_k,j_k)$ are the (almost surely unique) indices such that $\C{N}_{i_k,j_k}(t_k)\neq \C{N}_{i_k,j_k}(t_k-)$. 

%\JB{I would find it much clearer if you has a sequence of iid random variables $K_n, n\ge 1$ whos law depends on $\psi$ and which tells you who is killed at the $n$th event, a sequence $ A_n$ of iid uniform r.v. which tells you who is branching and then just a Poisson process of rate $Nr(t)$ }

Let $\C{W}:=((W_i(t))_{t\geq 0})_{i=1,2,\ldots,N}$ be a family of $N$ i.i.d Brownian motions. On each interval $(t_{m-1},t_m)$ for $m=1,2,\ldots$, the particle at location $X_i^N(t_{m-1})=\Theta^N_i(X^N(t_{m-1}))$ is driven by the Brownian motion $(W_i)_{t_{m-1}\leq t\leq t_m}$; so when $t\in (t_{m-1},t_m)$, the $N$ particles of the system are at locations $X^N_i(t_{m-1})+ W_i(t) - W_i(t_{m-1})$ for $i\in [N]$. Then at each time $t_{m}$, we move the particle at location $\Theta^N_{J_m}(X^N(t_m-))$ to the location of $\Theta^N_{I_m}(X^N(t_m-))$. Note that this does not exclude the possibility that $I_m=J_m$. Repeating this inductively, we can describe the process $(X^N(t))_{t\geq 0}$ as a function $X^N(t):=\Xi(\rho^N, \C{N},\C{W},\C{I},\C{J},t)$ for all times $t\geq 0$. Thus we have constructed the measure-valued process $(\mu^N_t)_{t\geq 0}$. %We consider $(\mu^N_t)_{t\in [0,T]}$ as a measure-valued process in the space $\C{D}([0,T],\C{M}_F^w(\B{R}))$, the space of cadlag paths $[0,T]\to \C{M}_F^w(\B{R})$, where $\C{M}_F^w(\B{R})$ is the space of finite measures on $\B{R}$ endowed with the weak topology. 

\subsection{Notation}
\begin{itemize}
    \item $H(x):=\is_{x\geq 0}$ denotes the Heaviside function and $\tilde{H}(x)=\is_{x\leq 0}$ its reflection.
    \item $\C{BC}^{2,1}(\B{R}\times [0,\infty),\B{R})$ is the set of functions $f$ such that $f_x$, $f_{xx}$, and $f_t$ exist, and $f$, $f_x$, $f_{xx}$, and $f_t$ are continuous and bounded, and $\C{BC}^{2,1}_0(\B{R}\times [0,\infty),\B{R})$ is the subset of $\C{BC}^{2,1}(\B{R}\times [0,\infty),\B{R})$ vanishing as the first variable goes to $\pm \infty$. $C_b(\B{R})$ will denote the set of continuous and bounded functions on $\B{R}$ of which the functions $f\in \C{BC}^{2,1}(\B{R}\times [0,\infty),\B{R})$ that are constant in the second variable form a dense subset. $C_c^{2,1}(\B{R}\times [0,\infty),\B{R})$ denotes the set of continuous functions $f:\B{R}\times [0,\infty)\to\B{R}$ with continuous derivatives $f_x$, $f_{xx}$, and $f_t$ and compact support.
    \item For function $f$ and measure $\mu$, $\langle f,\mu\rangle:=\int_{\B{R}} f(x)\mu(dx)$ and for functions $f,g$, $\langle f,g,\rangle :=\int_\B{R} f(x)g(x)dx$.
    \item For martingale $(M_t)_{t\geq 0}$, $([M]_t)_{t\geq 0}$ denotes the quadratic variation process of $M$. 
    \item For set $S$, $\C{P}(S)$ will denote the set of probability measures on $S$. 
    %\item For $N\in \B{N}$, $[N]$ will denote the set $\{1,2,\ldots,N\}$.
    \item $\C{M}_F^w(\B{R})$ will denote the space of finite measures on $\B{R}$ under the weak topology. Moreover, we write $\bar{\B{R}}$ for the extended real line $\B{R}\cup\{+\infty,-\infty\}$ which is the two-point compactification of $\B{R}$, and $\C{M}_F^w(\bar{\B{R}})$ for the corresponding space of measures.
    \item We will write $\mu \ll \lambda$ to denote that the measure $\mu$ is absolutely continuous with respect to the Lebesgue measure $\lambda$. 
\end{itemize}

\section{Main results}

The first result we prove here is the hydrodynamic limit for the $(\psi,r,N)$-BBM process under the conditions that $r$ is continuous and bounded and $\psi$ is positive, and continuous on $[0,1]$, with $\int_0^1 \psi(u)du=1$. We will show, in a sense which we will make clear, that the hydrodynamic limit of the process $(\psi,r,N)$-BBM is the unique classical solution of the PDE:
\begin{align} \label{mainResultPDE} U_t(x,t)= \frac12 U_{xx}(x,t) + r(t)\left(U(x,t) - \int_{1-U(x,t)}^1 \psi(s)ds\right) \end{align}

It is a classical result due to Kolmogorov, Petrovskii, and Piscunov (\S 3, Theorem 1, \cite{kpp}) that if $Y:\B{R}\times [0,\infty)\times \B{R}\to \B{R}$ is continuous, bounded, and Lipschitz continuous in the first and third variable, and $U(x,0)$ is bounded and continuous, then $U_t=U_{xx}+Y(x,t,U)$ has a unique classical solution. Note that if $r$ is continuous and bounded on $[0,\infty)$ and $\psi$ is positive and continuous on $[0,1]$, then $$Y(x,t,U):=r(t)\left(U-\int_{1-U}^1\psi(s)ds\right)\is_{U\in [0,1]}$$ immediately satisfies these conditions. 

Let us now define the assumption \textit{(A1)} on $\psi$ and $r$ under which our first main result holds. 

\textbf{\textit{(A1)}}: $\psi$, and $r$ satisfy assumption \textit{(A1)} if $\psi:[0,1]\to[0,\infty)$ is continuous with $\int_0^1 \psi(s)ds=1$ and $r$ is continuous and bounded.

\vspace{1em}

\begin{theorem} \label{psirNGeneralLimit}
    Let $(\mu^N_t)_{t\geq 0}\in \C{D}([0,T],\C{M}_F^w(\B{R}))$ be the empirical measure-valued process corresponding to the $(\psi,r,N)$-BBM process with $\mu^N_0=\rho^N$. Let $Q^N_T\in \C{P}(\C{D}([0,T],\C{M}^w_F(\B{R})))$ denote the distribution of $(\mu_t^N)_{t\in [0,T]}$. Then under assumption (A1), for any fixed $T\in (0,\infty)$, $Q^N_T$ converges weakly to $Q^\infty_T=\delta_{u}$, where $\delta_{u}$ is a Dirac mass on $u(x,t)\in \C{D}([0,T],\C{M}^w_F(\B{R}))$, and $U(x,t):=\int_x^\infty u(y,t)dy$ is the unique classical solution to \eqref{mainResultPDE} with initial condition $U(x,0)=\int_x^\infty \rho(dy)$.
\end{theorem}

%\JB{You need to say what classical solution means, in particular for the initial condition}
Next we define assumption \textit{(A2)}, which is a sufficient condition on $\psi$ and $r$ for the $(\psi,r,N)$-BBM process to have an asymptotic velocity $v_N$.

\textbf{\textit{(A2)}}: $\psi$ and $r$ satisfy assumption \textit{(A2)} if $\psi(x)=0$ for $x\in [1-p,1]$ for some $0<p<1$ and $\psi(x)>\epsilon$ for $x\in [0,\epsilon]$ for some $\epsilon>0$, and $r\equiv 1$.

By coupling our system to the $N$-BBM process, we can show that $v_N$ agrees with the speed of the $N$-BBM  up to order $(\log N)^{-2}$.

\vspace{1em}

\begin{theorem} \label{asympSpeedTheorem}
    Let $(X^N(t))_{t\geq 0}$ be a $(\psi,r,N)$-BBM with any initial configuration. Then under the assumption \textit{(A2)}, the $(\psi,r,N)$-BBM has asymptotic velocity:
    $$\lim_{t\to\infty}\frac{\Theta^N_1(X^N(t))}{t}=\lim_{t\to\infty}\frac{\Theta^N_N(X^N(t))}{t} = v_N = \sqrt{2}-\frac{\pi^2}{\sqrt{2}(\log N)^2} + o\left(\frac{1}{(\log N)^2}\right).$$ 
\end{theorem}

We may reasonably ask what happens if assumption \textit{\textbf{(A2)}} fails. If the assumption that $\psi(x)>0$ for $x\in [0,\epsilon]$ is broken, then we never kill the leftmost particles of the $(\psi,r,N)$-BBM, and accordingly, we would expect $\Theta_1^N(X^N(t))\to -\infty$ as $t\to\infty$ so that no positive asymptotic speed exists. One the other hand, if $\psi(x)>0$ for $x\in [1-p,1]$, then the rightmost $\lfloor pN\rfloor$ particles are killed at strictly positive rate, and we would expect this to slow down the rightmost particle so that $\lim_{N\to\infty}v_N<\sqrt{2}$. We do in fact expect that if $\psi(x)\to 0$ as $x\to 1$ sufficiently quickly, then there is still an asymptotic velocity close to $\sqrt{2}$ (see Conjecture \ref{speedConjecture}), however the specific method of our proof does not apply in this setup. 

In order to form a \textit{weak selection principle}, we also wish to relate $\sqrt{2}=\lim_{N\to\infty}v_N$ to the travelling wave speed or spreading speed of the PDE \eqref{mainResultPDE}. This is made precise in Section \ref{travellingSection}, and summarised in the following Theorem.

\vspace{1em}

\begin{theorem} \label{partialWeakSelecTheorem}
    Under the assumptions \textit{(A1), (A2)}, the minimal spreading speed (in a sense which we will make precise) of the PDE (\ref{mainResultPDE}) is $\sqrt{2}$. If, further, we have that $\int_{1-x}^1 \psi(s)ds \leq x$ for all $x\in [0,1]$, then the PDE (\ref{mainResultPDE}) has a travelling wave with speed $c$ for any $c\geq \sqrt{2}$. 
\end{theorem}

Together, Theorems \ref{asympSpeedTheorem} and \ref{partialWeakSelecTheorem} give a partial weak selection principle.

\section{Hydrodynamic limit}

In this section, we will prove the hydrodynamic limit result of Theorem \ref{psirNGeneralLimit}. The following gives an outline of the proof strategy, which was inspired by Demircigil and Tomasevic \cite{demircigilTomasevic}.
\begin{enumerate}
    \item Prove that the sequence of measure-valued processes is tight, so that the sequence has subsequential limits (Proposition \ref{tightness}).
    \item Show that if $\mu^\infty_t$ is a subsequential limit, then $\mu^\infty_t([x,\infty))$ is a weak solution of the PDE \eqref{mainResultPDE} (Propositions \ref{weakRepPropn}, \ref{muinfLimit}, \ref{cdfPdfEquiv}). 
    \item Show that the PDE has a unique weak solution (Proposition \ref{Uniqueness}).
\end{enumerate} 

First we will give a probabilistic representation of $\mu^N$ tested against a test functions in $\C{BC}^{2,1}$. We allow the test functions to depend on time as this will allow us to easily prove uniqueness of weak solutions in step 3 (Proposition \ref{Uniqueness}).

\vspace{1em}

\begin{propn} \label{weakRepPropn}
    Let $(\mu_t^N)_{t\geq 0}$ be the measure-valued Markov process corresponding to the $(\psi,r,N)$-BBM process. Let $f(x,t)\in \C{BC}^{2,1}(\B{R}\times [0,\infty),\B{R})$. Then $\mu^N$, and $f$ satisfy the following equation:
    \begin{align}\label{weakRep}\langle \mu_t^N,f(\cdot,t)\rangle = \langle \mu_0^N, f(\cdot,0)\rangle &+ \int_0^t \langle \mu_s^N, \frac12 f_{xx}(\cdot,s) + f_t(\cdot,s) + r(s)f(\cdot,s)\rangle \\
    \nonumber&-\left\langle \mu_s^N, Nf(\cdot,s)r(s)\int_{(H\star \mu^N_s)(\cdot)-1/N}^{(H\star \mu^N_s)(\cdot)}\psi(u)du\right\rangle ds + M_t^{N,W} + M_t^{N,P},\end{align}
    where $M^{N,W}_t$ is a continuous local martingale with $\B{E}[[M^{N,W}]_t]\xrightarrow[N\to\infty]{}0$ for fixed $t$, and $M^{N,P}_t$ is a local martingale with $\B{E}[[M^{N,P}]_t]\xrightarrow[N\to\infty]{}0$ for fixed $t$.
\end{propn}

\begin{proof} Note that for an atomic measure such as $\mu^N$ we have $\langle \mu^N_t, f(\cdot,t)\rangle = \frac1N \sum_{i=1}^N f(X^N_i(t),t)$. Then by Ito's formula, for $t\in [0,t_1)$:
$$f(X^N_i(t),t)=f(X^N_i(0),0)+\int_0^t f_x(X^N_i(s),s)dW_i(s)+\int_0^t f_t(X^N_i(s),s)+\frac12 f_{xx}(X^N_i(s),s)ds,$$
so for $t\in [0,t_1)$
$$\langle \mu_t^N, f(\cdot,t)\rangle = \langle \mu_0^N, f(\cdot,0)\rangle + \int_0^t \langle \mu_s^N, f_t(\cdot,s) + \frac12 f_{xx}(\cdot,s)\rangle ds + \frac1N \sum_{i=1}^N \int_0^t f_x(X^N_i(s),s)dW_i(s).$$

Then at the time $t_1$, the discontinuity in $\C{N}$ induces a change in $\langle \mu_t^N, f(\cdot,t)\rangle$ (unless $I_1=J_1$). Specifically, $\langle\mu_{t_1-}^N,f(\cdot,t_1-)\rangle$ jumps  by $\frac1N(f(\Theta^N_{I_1}(X^N(t_1-)),t_1)-f(\Theta^N_{J_1}(X^N(t_1-)),t_1))$. Thus
\begin{align*}\langle \mu_{t_1}^N, f(\cdot,t_1)\rangle = \langle \mu_0^N, f(\cdot,0)\rangle &+ \frac1N \sum_{i=1}^N \int_0^{t_1} f_x(X^N_i(s),s)dW_i(s) +\int_0^{t_1} \langle \mu_s^N, f_t(\cdot,s) + \frac12 f_{xx}(\cdot,s)\rangle ds \\
&+ \frac1N \Big(f(\Theta^N_{I_1}(X^N(t_1-)),t_1)-f(\Theta^N_{J_1}(X^N(t_1-)),t_1)\Big),\end{align*}
and therefore, for general $t$, integrating with respect to $\C{N}(t)$, we can write
\begin{align}\begin{split}\label{randomIntegralRep}\langle \mu_{t}^N, f(\cdot,t)\rangle = \langle \mu_0^N, f(\cdot,0)\rangle &+ M^{N,W}_t + \int_0^t \langle \mu_s^N, f_t(\cdot,s) + \frac12 f_{xx}(\cdot,s)\rangle ds \\
&+\frac1N \int_0^t f(\Theta_{I_{\C{N}(s)}}^N(X^N(s-)),s)-f(\Theta_{J_{\C{N}(s)}}^N(X^N(s-)),s) \C{N}(ds),\end{split}\end{align}
where $M^{N,W}_t$ is the continuous local martingale $M^{N,W}_t:=N^{-1}\sum_{i=1}^N \int_0^t f_x(X^N_i(s),s)dW_i(s)$. Now note the process $\C{N}(t)-N\int_0^t r(s)ds$ is a compensated Poisson process and thus a martingale %\JB{remove: (see for example the Example at the end of I.7, \cite{protter})}.
Then we wish to show that the process 
\begin{align*}M^{N,P}_t:=\frac1N \int_0^t f(&\Theta^N_{I_{\C{N}(s)}}(X^N(s-)),s)-f(\Theta^N_{J_{\C{N}(s)}}(X^N(s-)),s)\C{N}(ds)\\
&-\frac1N \sum_{i,j\in [N]^2}\int_0^t \big(f(\Theta^N_i(X^N(s-)),s)-f(\Theta^N_j(X^N(s-)),s)\big)\int_{(i-1)/N}^{i/N}\psi(u)du\,r(s)ds\end{align*}
is a local martingale with respect to the filtration $\C{F}_s:=\sigma(W_i(u),\C{N}(u),I_m,J_m:i\in [N],u\leq s, m\leq \C{N}(s))$. This follows from that fact that the random variables $I_{\C{N}(s)}$ and $J_{\C{N}(s)}$ are independent of $\C{N}$, so that:
\begin{align*}
    \B{E}\Bigg[\int_s^t f(\Theta^N_{I_{\C{N}(u)}}&(X^N(u-)),u)-f(\Theta^N_{J_{\C{N}(u)}}(X^N(u-)),u)\C{N}(du) | \C{F}_s\Bigg]\\
    &=\B{E}\left[\int_s^t f(\Theta^N_{I_{\C{N}(u)}}(X^N(u-)),u)-f(\Theta^N_{J_{\C{N}(u)}}(X^N(u-)),u)\C{N}(du)\right] \\
    &=\B{E}\left[\sum_{i,j\in [N]^2}\int_s^t \big(f(\Theta^N_i(X^N(u-)),u)-f(\Theta^N_j(X^N(u-)),u)\big)\B{P}(I_{\C{N}(u)}=i,J_{\C{N}(u)}=j)\C{N}(du)\right] \\
    &= \B{E}\left[\sum_{i,j\in [N]^2}\int_s^t \big(f(\Theta^N_i(X^N(u-)),u)-f(\Theta^N_j(X^N(u-)),u)\big)\times \frac1N\int_{(i-1)/N}^{i/N}\psi(v)dv \times Nr(u)du\right]%\\
    %&= \sum_{i,j\in [N]^2}\int_s^t \big(f(\Theta^N_i(X^N(u-)),u)-f(\Theta^N_j(X^N(u-)),u)\big)\int_{(i-1)/N}^{i/N}\psi(v)dv r(u)du,
\end{align*}
where the final equality follows from the fact that $C(t):=\C{N}(t)-\int_0^t Nr(s)ds$ is a martingale and hence $\int_0^t \big(f(\Theta^N_i(X^N(s-)),s)-f(\Theta^N_j(X^N(s-)),s)\big)\frac1N \int_{(i-1)/N}^{i/N}\psi(u)du C(ds) $ is a martingale. From this it follows that $\B{E}[M^{N,P}_t-M^{N,P}_s|\C{F}_s]=0$, thus confirming that $M^{N,P}_t$ is a local martingale. So equation (\ref{randomIntegralRep}) can be written as 
\begin{align*}\langle \mu_t^N, f(\cdot,t)\rangle = &\langle \mu_0^N,f(\cdot,0) \rangle + M_t^{N,W} + M_t^{N,P} + \int_0^t \langle \mu_s^N, f_t(\cdot,s) + \frac12 f_{xx}(\cdot,s)\rangle ds \\
&+ \frac1N \sum_{(i,j)\in [N]^2}\int_0^t \big(f(\Theta^N_i(X^N(s-)),s)-f(\Theta^N_j(X^N(s-)),s)\big)\int_{(i-1)/N}^{i/N}\psi(u)du\,r(s)\, ds\end{align*}

We will now deal with the final term of the above equation - the sum over $i,j\in [N]^2$ - in two parts. Using the observation that $N(H\star \mu^N_s)(\Theta^N_i(X^N(s)))=N\int_{-\infty}^{\Theta^N_i(X^N(s))}d\mu^N_s=i$ for all $s$, and summing first over $j\in [N]$, we can calculate that
    \begin{align*}\frac1N \sum_{(i,j)\in [N]^2}\int_0^t  f(\Theta^N_i(X^N(s&-)),s)r(s)\int_{(i-1)/N}^{i/N}\psi(u)du\,ds \\
    &= \sum_{i\in [N]} \int_0^t f(\Theta^N_i(X^N(s-)),s)r(s)\int_{(H\star \mu_s^N)(\Theta^N_i(X^N(s)))-1/N}^{(H\star \mu_s^N)(\Theta^N_i(X^N(s)))}\psi(u)du \,ds \\
    &=\int_0^t \left\langle \mu_s^N, Nf(\cdot,s)r(s)\int_{(H\star \mu^N_s)(\cdot)-1/N}^{(H\star \mu^N_s)(\cdot)}\psi(u)du\right\rangle, \end{align*}
%where the factor of $N$ comes from summing over $j\in [N]$. 
Then observing that $\sum_{i\in [N]}\int_{(i-1)/N}^{i/N}\psi(u)du = \int_0^1 \psi(u)du = 1$, and summing first over $i\in [N]$, we have
\begin{align*}\frac1N \sum_{(i,j)\in [N]^2}\int_0^t &f(\Theta^N_j(X^N(s-)),s)r(s)\int_{(i-1)/N}^{i/N}\psi(u)du\, ds\\
&=\int_0^t \frac1N \sum_{j\in [N]}f(\Theta^N_j(X^N(s-)),s)r(s)ds = \int_0^t \langle \mu_s^N, f(\cdot,s)r(s)\rangle ds.\end{align*}
Putting this together yields equation (\ref{weakRep}), as required. Note that by Ito's isometry and the independence of the $W_i$'s, $M_t^{N,W}$ has expected quadratic variation 
\begin{align} \label{qvW} \B{E}[[M^{N,W}]_t]=\B{E}\left[\left[ \frac1N \sum_{i=1}^N \int_0^t f_x(X^N_i(s),s)dW_i(s)\right]_t\right] \leq \frac{1}{N^2}\sum_{i=1}^N \int_0^t \|f_x\|_\infty^2 ds= \frac{\|f_x\|_\infty^2 t}{N}. \end{align}
Since $f_x$ is bounded, thus $\B{E}[[M^{N,W}]_t]\to 0$ as $N\to\infty$ for any fixed $t$. Furthermore, by Ito's isometry, the expected quadratic variation of $M^{N,P}_t$ can be calculated:
\begin{align} \label{qvP}
    \nonumber \B{E}[[M^{N,P}]_t] &= \B{E}\left[\left[\frac1N\int_0^t \big(f(\Theta^N_{I_{\C{N}(s)}}(X^N(s-)),s)-f(\Theta^N_{J_{\C{N}(s)}}(X^N(s-)),s)\big)\C{N}(ds)\right]_t\right]\\
    \nonumber &=\B{E}\left[\int_0^t \frac{1}{N^2}\big(f(\Theta^N_{I_{\C{N}(s)}}(X^N(s-)),s)-f(\Theta^N_{J_{\C{N}(s)}}(X^N(s-)),s)\big)^2 \C{N}(ds)\right] \\
    &\leq \B{E}\left[\int_0^t \frac{4\|f\|_\infty^2}{N^2}\C{N}(ds)\right] = \frac4N \|f\|_\infty^2 \int_0^t r(s)ds
\end{align}
Since $f$ and $r$ is bounded, thus $\B{E}[[M^{N,P}]_t]\xrightarrow[N\to\infty]{}0$ for fixed $t$. This concludes the proof. 
\end{proof}

\vspace{1em}

%Since we are seeking the hydrodynamic limit of the $(\psi,r,N)$-BBM system, we wish to take the limit as $N\to\infty$ of the sequence $((\mu^N_t)_{t\geq 0})_{N=1,2,\ldots}$ of measure-valued processes. Therefore in order to prove that such limits exist, we prove tightness of the sequence in the appropriate space. 
Next, we prove tightness of the sequence of laws of $(\mu_t^N)_{t\geq 0}$. Specifically, we will prove that the sequence of laws is tight in the space $\C{P}(\C{D}([0,T],\C{M}^w_F(\bar{\B{R}}))$, where $\bar{\B{R}}$ is the extended real line $\B{R}\cup\{-\infty,+\infty\}$. The benefit of this method is that $\bar{\B{R}}$ is compact, and as a consequence $\{\mu:\mu(\bar{\B{R}})=1\}$ is a compact subset of $\C{M}^w_F(\bar{\B{R}})$. This will help us to prove a compactification condition required to show tightness. 

\vspace{1em}

\begin{propn} \label{tightness}
    Fix $T>0$, and let $Q^N_T$ denote the law of $(\mu_t^N)_{t\in [0,T]}$. Then the sequence of $(Q^N_T)_{N=1,2,\ldots}$ is tight in $\C{P}(\C{D}([0,T],\C{M}_F^w(\bar{\B{R}})))$ with respect to the Skorokhod metric.
\end{propn}

\begin{proof}
    In order to show that the sequence $((\mu_t^N)_{t\geq 0})_{N=1,2,\ldots}$ is tight in $\C{P}(\C{D}([0,T],\C{M}^w_F(\bar{\B{R}})))$ with respect to the Skorokhod topology, it is sufficient (Theorem 1.18, \cite{etheridge}) to show the following conditions:
    \begin{enumerate}[(i)]
        \item \textit{Compact containment:} For all $\epsilon > 0$ there exists a compact subset $K_\epsilon\subseteq \C{M}_F^w(\bar{\B{R}})$ such that $$\inf_{N\in \B{N}} \B{P}(\mu_t^N \in K_\epsilon\;  \forall \, t\in [0,T])>1-\epsilon$$
        \item \textit{Tightness of real-valued processes:} The sequence of laws of $(\langle \mu_t^N,f\rangle)_{0\leq t\leq T}$ is tight in $\C{P}(\C{D}([0,T],\B{R}))$ for every function $f$ in a dense subset of $C_b(\bar{\B{R}})$
    \end{enumerate}
    
    Observe that the subset $\{\mu:\mu(\bar{\B{R}})=1\}\subseteq \C{M}_F^w(\bar{\B{R}})$, the subset consisting of probability measures, is a compact subset since $\bar{\B{R}}$ is itself compact. This is easily seen by observing that any subset of $\{\mu:\mu(\bar{\B{R}})=1\}$ is tight, since $\bar{\B{R}}$ is a compact subset of $\bar{\B{R}}$ such that $\mu(\bar{\B{R}})>1-\epsilon$ for all $\mu$. Then since $\mu_t^N(\B{R})=1$ for all $t\in [0,T]$ and $N\in \B{N}$, the compact containment condition (i) holds immediately. 
    
    We now prove condition (ii). Since each $\mu_t^N$ has support on $\B{R}$, thus it is sufficient to show that $((\langle \mu_t^N,f\rangle)_{t\geq 0})_{N=1,2,\ldots}$ is tight for all functions in a dense subset of $C_b(\B{R})$. We choose $\C{BC}^{2,1}(\B{R}\times [0,\infty),\B{R})\cap C_b(\B{R})$ to be our dense subset of $C_b(\B{R})$. By Aldous' tightness criterion (see, for example, Theorem 16.10, \cite{billingsley}), $((\langle \mu_t^N,f\rangle)_{t\geq 0})_{N=1,2,\ldots}$ is tight if the following conditions hold:
    \begin{enumerate}[A]
        \item For every $m>0$, $\lim_{a\to\infty}\limsup_{N\to\infty}\B{P}(\sup_{0\leq t\leq m}|\langle \mu_t^N,f\rangle|\geq a)=0$.
        \item For each $\epsilon, \eta, m >0$, there exists $\delta_0, N_0$ such that if $\delta \leq \delta_0$, $N\geq N_0$, and $\tau$ is a stopping time such that $\tau\leq m$, then $\B{P}(|\langle \mu_{\tau+\delta}^N,f\rangle-\langle \mu^N_{\tau},f\rangle|\geq \epsilon)\leq \eta$.
    \end{enumerate}

    Condition A follows immediately by observing that $|\langle \mu_t^N, f\rangle| \leq \|f\|_\infty$, therefore since $f$ is bounded, then for $a\geq \|f\|_\infty$, this probability is $0$. Now let us prove condition B. Since $f\in \C{BC}^{2,1}(\B{R}\times [0,\infty),\B{R})\cap C_b(\B{R})$, thus $f_t=0$, and $f$, $f_x$, and $f_{xx}$ are all bounded for $t\in [0,m+\delta]$, say by a constant $C_{f}$. So \eqref{weakRep} gives that
    \begin{align*}
    |\langle \mu_{\tau+\delta}^N,f\rangle - \langle \mu_\tau^N,f\rangle| %= \Bigg|\int_\tau^{\tau+\delta}\langle \mu_s^N,\frac12 f_{xx}+ r(s)f\rangle \\ &- \left\langle \mu^N_s, Nfr(s)\int_{(H\star \mu^N_s)(\cdot)-1/N}^{(H\star \mu^N_s)(\cdot)}\psi(u)du \right\rangle ds + (M_{\tau+\delta}^{N,W}-M_{\tau}^{N,W}) + (M_{\tau+\delta}^{N,P}-M_{\tau}^{N,P})\Bigg| \\
    \leq  (1+C_r+C_rC_\psi)C_f\delta+|M_{\tau+\delta}^{N,W}-M_{\tau}^{N,W}| + |M_{\tau+\delta}^{N,P}-M_{\tau}^{N,P}|
    \end{align*}
    %Since $f\in \C{BC}^{2,1}(\B{R}\times [0,\infty),\B{R})\cap C_b(\B{R})$, thus $f$, $f_x$, and $f_{xx}$ are all bounded for $t\in [0,m+\delta]$, say by a constant $C_{f}$. Moreover, $r$ and $\psi$ are bounded, say by constants $C_r$ and $C_\psi$ respectively, therefore:
    %\begin{align*}\Bigg|\int_\tau^{\tau+\delta}\langle \mu_s^N,\frac12f_{xx} &+ r(s)f\rangle - \Big\langle \mu^N_s, Nfr(s)\int_{(H\star \mu^N_s)(\cdot)-1/N}^{(H\star \mu^N_s)(\cdot)}\psi(u)du \Big\rangle ds\Bigg|\\ &\leq \int_{\tau}^{\tau+\delta}(\frac12+C_r)C_f + C_rC_\psi C_fds \leq (1+C_r+C_rC_\psi)C_f\delta.\end{align*}

    %Then by the Burkholder-Davis-Gundy inequality, and the bounds given in the proof of Theorem \ref{weakRepPropn}, there exists a constant $C$ such that for any $0<s<t\leq m+\delta$, 
    %\begin{align}\label{WmartConv}\B{E}[|M_{t}^{N,W}-M_{s}^{N,W}|^2]\leq C\B{E}[[M^{N,W}]_t]\leq \frac{C\|f_x\|_\infty^2 (t-s)}{N}\leq \frac{CC_f^2(t-s)}{N},\end{align}
    %and
    %\begin{align}\label{PmartConv}\B{E}[|M_t^{N,P}-M_s^{N,P}|^2]\leq C\B{E}[[M^{N,P}]_t]\leq \frac{4C\|f\|_\infty^2}{N}\int_s^t r(s)ds \leq \frac{4CC_f^2C_r(t-s)}{N}.\end{align}
    %Therefore by Markov's inequality,
    %$$\B{P}\big(|M_{\tau+\delta}^{N,W}-M_{\tau}^{N,W}|^2>\epsilon^2/9\big)\leq \frac{9CC_f^2\delta}{N\epsilon^2}, \quad \B{P}\big(|M_{\tau+\delta}^{N,P}-M_\tau^{N,P}|^2>\epsilon^2/9\big)\leq \frac{36CC_f^2C_r\delta}{N\epsilon^2}.$$
    %Hence choosing $N_0$ and $\delta_0$ such that $(1+C_r+C_r C_\psi)C_f\delta_0 < \epsilon/3$, $9CC_f^2\delta_0/N_0\epsilon^2 <\eta/2$, and $36CC_f^2C_r\delta_0/N_0\epsilon^2 < \eta/2$, then:
    Therefore by Markov's inequality and the Burkholder-Davis-Gundy inequality, there exists a constant $C$ such that:
    \begin{align*}
        \B{P}(|\langle \mu_{\tau+\delta}^N,&f\rangle-\langle \mu_\tau^N,f\rangle|> \epsilon) \leq \B{P}\Big((1+C_r+C_r C_\psi)C_f\delta+|M_{\tau+\delta}^{N,W}-M_\tau^{N,W}|+|M_{\tau+\delta}^{N,P}-M_{\tau}^{N,P}|>\epsilon\Big)\\
        &\leq \is_{\{(1+C_r+C_rC_\psi)C_f\delta\geq \epsilon/3\}}+\B{P}(|M_{\tau+\delta}^{N,W}-M_\tau^{N,W}|>\epsilon/3)+\B{P}(|M_{\tau+\delta}^{N,P}-M_{\tau}^{N,P}|>\epsilon/3) \\
        &\leq \is_{\{(1+C_r+C_rC_\psi)C_f\delta\geq \epsilon/3\}} + \frac{9}{\epsilon^2}\B{E}[|M_{\tau+\delta}^{N,W}-M_\tau^{N,W}|^2]+\frac{9}{\epsilon^2}\B{E}[|M_{\tau+\delta}^{N,P}-M_\tau^{N,P}|^2]\\
        &\leq \is_{\{(1+C_r+C_rC_\psi)C_f\delta\geq \epsilon/3\}} + \frac{9C}{\epsilon^2}\B{E}[[M^{N,W}]_\delta]+\frac{9C}{\epsilon^2}\B{E}[[M^{N,P}]_\delta]\xrightarrow[N\to\infty,\delta\to 0]{}0\\
    \end{align*}
    Therefore condition B follows, hence the sequence $(\langle \mu_t^N,f\rangle)_{t\geq 0})_{N=1,2,\ldots}$ is tight, thus proving condition (ii), and completing the proof. 
\end{proof}

A very similar method to the above is employed by Etheridge (Section 1.4, \cite{etheridge}). To prove the compact containment condition, they instead uses the one-point compactification of $\B{R}$, $\hat{\B{R}}$, instead of $\bar{\B{R}}$ (the two-point compactification of $\B{R}$). However, since $+\infty$ and $-\infty$ are not distinguished in $\hat{\B{R}}$, continuous bounded functions $f\in C_b(\hat{\B{R}})$ must have $\lim_{x\to-\infty}f(x)=\lim_{x\to+\infty}f(x)$. However we wish to have convergence of $\langle \mu_t^N, f(\cdot,t)\rangle$ for functions $f:\bar{\B{R}}\times [0,\infty)\to \B{R}$ with $f(-\infty,t)\neq f(+\infty,t)$, which is why we use $\C{M}^w_F(\bar{\B{R}})$ instead. 

This generalisation from $\C{M}_F^w(\B{R})$ to $\C{M}_F^w(\bar{\B{R}})$ will not cause any additional difficulties for us. Since $(Q^N_T)_{N=1,2,\ldots}$ is tight, therefore by Prokhorov's theorem (Theorem 5.1, \cite{billingsley}), there is a weakly convergent subsequence $(Q^{N_k}_T)_{k=1,2,\ldots}$. By appealing to a comparison with branching Brownian motion, we can show that any subsequential limit $(\mu_t^\infty)_{t\geq 0}$ is absolutely continuous with respect to the Lebesgue measure and has $\mu_t^\infty(\{-\infty,\infty\})=0$ for all times $t\in [0,T]$ almost surely. %\JB{Not sure what this does: As a result, this means that the corresponding cumulative distribution function, $(H\star \mu_s^\infty)(x)$, is continuous in $x$,} and, by the Radon-Nikodym Theorem, this means that if $(\mu_s^\infty)_{s\in [0,T]}\sim Q^\infty_T$, then for any fixed $s\in [0,T]$, $\mu_t^\infty$ $Q^\infty_T$-a.s. has a density function $m_s(x)$ such that $\mu_s^\infty(A)=\int_A m_s(x)dx$ for any measurable set $A$. %This comparison also yields that any subsequential limit of $(\mu^N)_{N\geq 1}$ has exponentially decaying tails in the sense that the density function $m_s$ has $m_s(x)=O(e^{-|x|})$ as $|x|\to\infty$.

\vspace{1em}

\begin{lemma} \label{absContEveryTime} Let $Q^\infty_T$ be a subsequential limit of $(Q^N_T)_{N\geq 1}$ and let $(\mu^\infty_t)_{t\in [0,T]}\sim Q^\infty_T$. Then $\mu_t^\infty\ll\lambda$ and $\mu^\infty_t(\{-\infty,\infty\})=0$ for all $t\in [0,T]$ $Q_T^\infty$-a.s. 
\end{lemma}

This proof essentially follows from the fact that the $(\psi,r,N)$-BBM can be bounded above by a branching Brownian motion started from $N$ particles, which has a well known hydrodynamic limit which is absolutely continuous with respect to the Lebesgue measure for all $t\in [0,T]$. 

\begin{proof}
    We construct the $(\psi,r,N)$-BBM from a branching Brownian motion as follows. Consider a branching Brownian motion $(X_u(t):u\in \C{U}(t))_{t\geq 0}$ starting from $N$ particles according to the initial configuration $\rho^N=\sum_{i=1}^N \delta_{x_i^N}$ which branches at time-inhomogeneous rate $r$. Let $\C{U}(t)$ denote the set of particles at time $t$, and $\C{U}^i(t)$ be the subset of $\C{U}(t)$ consisting of the descendants of the particle which starts at location $x_i^N$. Now we will colour the particles of the BBM blue and red. Initially, all particles are blue. Then at each branching time of a blue particle, we colour exactly one blue particle red; the $j$\textsuperscript{th} leftmost blue particle being coloured red with probability $\int_{(j-1)/N}^{j/N}\psi(u)du$. Subsequently, blue particles always branch into blue particles and red into red. Clearly there are always exactly $N$ blue particles, and the subset of blue particles describes a $(\psi,r,N)$-BBM. 
    
    Now let $\hat{\mu}^N_t = \frac1N \sum_{i=1}^N\sum_{u\in \C{U}^i(t)}\delta_{X_u(t)}$ be the rescaled empirical measure of this coloured BBM process at time $t$. It is well-known (and proven in the Appendix, Theorem \ref{bbmHydroAppendix}, for completeness) that if $\hat{Q}_T^N$ is the law of $(\hat{\mu}_t^N)_{t\in [0,T]}$, then $\hat{Q}_T^N$ converges weakly to $\hat{Q}_T^\infty$, and if $\hat{\mu}^\infty \sim \hat{Q}^\infty_T$, then $\hat{\mu}_t$ has density $u(x,t)=\exp(\int_0^t r(s)ds)\frac{\partial}{\partial x}\B{P}_\rho(B(t)\leq x)$ $\hat{Q}^\infty_T$-a.s., and $u$ is the unique classical solution to the PDE $u_t=\frac12u_{xx}+r(t)u$. Therefore since $\hat{\mu}_t^\infty\ll \lambda$ and $\hat{\mu}_t^\infty(\{-\infty,\infty\})=0$ for all $t\in [0,T]$ $\hat{Q}^\infty_T$-a.s., and $\mu_t^N$ is dominated by $\hat{\mu}_t^N$ for all $N$, thus $\mu_t^\infty\ll \lambda$ and $\mu_t^\infty(\{-\infty,\infty\})=0$ for all $t\in [0,T]$ $Q^\infty_T$-a.s..
    %
    %Moreover, since the density of $\hat{\mu}^\infty_t$ is described by $u(\cdot,t):\B{R}\to \B{R}$, thus $\hat{\mu}^\infty_t$ has no mass at infinity (ie. $\hat{\mu}^\infty_t(\{-\infty,\infty\})=0$), and hence $\mu_t^\infty$ similarly has no mass at infinity $Q_T^\infty$-almost surely.
\end{proof}

Since we know that $\mu_t^\infty$ has no mass at $\infty$, convergence in the space $\C{D}([0,T],\C{M}^w_F(\bar{\B{R}}))$ implies convergence in the space $\C{D}([0,T],\C{M}^w_F(\B{R}))$, which is what we use from here onwards. Next we show, by Kolmogorov's continuity criterion, that the map $t\mapsto \mu_t^\infty$ is continuous on $[0,T]$ $Q^\infty_T$-a.s.. In particular:

\vspace{1em}

\begin{lemma} \label{asContinuity} Let $Q^\infty_T$ be a subsequential limit of $(Q^N_T)_{N=1,2,\ldots}$ and let $(\mu^\infty_t)_{t\in [0,T]}\sim Q^\infty_T$. Then the map $t\mapsto \mu_t^\infty$ is continuous on $[0,T]$ with respect to the weak topology $Q_T^\infty$-a.s..
\end{lemma}

\begin{proof}
    Since $\mu_t^\infty$ is cadlag, we will show using the Kolmogorov continuity theorem to show that there exists a continuous modification of $\mu_t^\infty$, and therefore as a result $\mu_t^\infty$ is almost surely continuous (see Theorem 1, \cite{schilling}). So fix $t_1,t_2\in [0,T]$ and $f\in \C{BC}^{2,1}(\B{R}\times [0,\infty),\B{R})$. By the definition of convergence in the Skorokhod topology on $\C{D}([0,T],\C{M}^w_F(\B{R}))$, there exists a sequence of functions $\lambda_N:[0,T]\to[0,T]$ such that $\sup_{t\in [0,T]}|\lambda_N(t)-t|$ and $\sup_{t\in [0,T]}|\langle f,\mu_{\lambda_N(t)}^N\rangle-\langle f, \mu_t^\infty\rangle|$ converge to $0$ as $N\to\infty$. Therefore by Fatou's lemma:
    \begin{align*}
        \B{E}[|\langle \mu_{t_1}^\infty,f\rangle -\langle \mu_{t_2}^\infty,f\rangle|^2]&=\B{E}[\lim_{N\to\infty}|\langle \mu_{\lambda_N(t_1)}^N,f\rangle -\langle \mu_{\lambda_N(t_2)}^N,f\rangle|^2]\leq \liminf_{N\to\infty}\B{E}[|\langle \mu_{\lambda_N(t_1)}^N,f\rangle -\langle \mu_{\lambda_N(t_2)}^N,f\rangle|^2]
        %&\leq \liminf_{N\to\infty}\Bigg\{\left(\int_{t_1}^{t_2}\langle \mu_s^N,\frac12 f_{xx}(\cdot,s)+f_t(\cdot,s)+r(s)f(\cdot,s)\rangle - \langle \mu_s^N,Nf(\cdot,s)
    \end{align*}
    As in the proof of Proposition \ref{tightness}, using the representation \eqref{weakRep}, we can bound
    \begin{align*}
        |\langle \mu_{\lambda_N(t_1)}^N,f\rangle -\langle \mu_{\lambda_N(t_2)}^N,f\rangle|^2 \leq C|\lambda_N(t_1)-\lambda_N(t_2)|^2+9|M_{\lambda_N(t_1)}^{N,W}-M_{\lambda_N(t_2)}^{N,W}|^2+9|M_{\lambda_N(t_1)}^{N,P}-M_{\lambda_N(t_2)}^{N,P}|^2
    \end{align*}
    for some constant $C$. Then by \eqref{qvW} and \eqref{qvP} and the Burkholder-Davis-Gundy inequality, $\B{E}[|M_{\lambda_N(t_1)}^{N,W}-M_{\lambda_N(t_2)}^{N,W}|^2]$ and $\B{E}[|M_{\lambda_N(t_1)}^{N,P}-M_{\lambda_N(t_2)}^{N,P}|^2]$ converge to $0$ as $N\to\infty$, therefore:
    \begin{align*}
        \liminf_{N\to\infty}\B{E}[|\langle \mu_{\lambda_N(t_1)}^N,f\rangle &-\langle \mu_{\lambda_N(t_2)}^N,f\rangle|^2]\leq \liminf_{N\to\infty}C|\lambda_N(t_1)-\lambda_N(t_2)|^2 \\
        &= \liminf_{N\to\infty}9C(|\lambda_N(t_1)-t_1|^2+|t_1-t_2|^2+|t_2-\lambda_N(t_2)|^2)=9C|t_1-t_2|^2,
    \end{align*}
    since $\sup_{t\in [0,T]}|\lambda_N(t)-t|\xrightarrow[N\to \infty]{}0$. Therefore $\langle \mu_t^\infty,f\rangle$ satisfies the conditions of the Kolmogorov continuity theorem, therefore $\langle \mu_t^\infty,f\rangle$ has a continuous modification, and thus is almost surely continuous (Theorem 1, \cite{schilling}). Therefore $\langle \mu_t^\infty,f\rangle$ is continuous for all $f$ in a countable dense subset of $C_b(\B{R})$ almost surely, and hence $t\mapsto\mu_t^\infty$ is continuous with respect to the weak topology almost surely. 
\end{proof}

Next we prove that each subsequential limiting distribution is concentrated on the weak solutions of a deterministic PDE. 

\vspace{1em}
\begin{propn} \label{detConvPropn}
    Let $Q_T^\infty$ be a sub-sequential limit of $(Q^N_T)_{N=1,2,\ldots}$, and let $(\mu_t^\infty)_{t\in [0,T]}\sim Q^\infty_T$. Then for any $f\in \C{BC}^{2,1}(\B{R}\times [0,\infty),\B{R})$,
    $\mu^\infty$ satisfies $Q_T^\infty$-a.s. the equation
    \begin{align} \label{muinfLimit}
        \langle \mu_t^\infty, f(\cdot,t)\rangle = \langle \rho, f(\cdot,0)\rangle + \int_0^t \langle \mu_s^\infty, \frac12 f_{xx}(\cdot,s) &+ f_t(\cdot,s) + r(s)f(\cdot,s)\rangle - \langle \mu_s^\infty,f(\cdot,s)r(s)\psi((H\star \mu_s^\infty)(\cdot))\rangle ds.
    \end{align}
\end{propn}

In particular, \eqref{muinfLimit} is the weak formulation of the non-linear equation:
\begin{align*}
    u_t(x,t) = \frac12 u_{xx}(x,t) + r(t)u(x,t)\left(1 -\psi\left(\int_{-\infty}^x u(y,t)dy\right)\right)
\end{align*}
which describes the limiting density of the process. We will subsequently transform this into its integrated version, the PDE \eqref{mainResultPDE} which describes the limiting cumulative distribution function.

\begin{proof}
    Without loss of generality, throughout this proof, we will consider the sub-sequential limit to be labelled $N=1,2,\ldots$. For measure-valued process $\mu = (\mu_t)_{t\geq 0}$, function $f(x,t)$, integer $N$, and time $t$, define:
    \begin{align*}G(\mu,f,N,t):=\langle \mu_t,f(\cdot,t)\rangle - \langle \mu_0,f(\cdot,0)\rangle - \int_0^t \langle \mu_s,&\frac12 f_{xx}(\cdot,s) + f_t(\cdot,s) + r(s)f(\cdot,s)\rangle \\
    &- \big\langle \mu_s,Nf(\cdot,s)r(s)\int_{(H\star \mu_s)(\cdot)-1/N}^{(H\star \mu_s)(\cdot)}\psi(u)du\big\rangle ds.\end{align*}

    Therefore by Proposition \ref{weakRepPropn}, $G(\mu^N,f,N,t)=M_t^{N,W}+M_t^{N,P}$. Now $|M_t^{N,W}+M_t^{N,P}|^2 \leq 4(|M_t^{N,W}|^2+|M_t^{N,P}|^2)$. Therefore by \eqref{qvW}, \eqref{qvP}, and the Burkholder-Davis-Gundy inequality, for every fixed $t\geq 0$:
    \begin{align}\label{EGto0}\B{E}[G(\mu^N,f,N,t)^2]\xrightarrow[N\to\infty]{}0.\end{align}

    %\textcolor{red}{I think the following argumentds are not right - they are treating the convergence as referring to the measures $\mu^N_t$ to $\mu_t^\infty$, but actually it is the measure on this space of measures... Instead we define a set of probability one on which this weak convergence of the measure $\mu_t^N$ to $\mu_t^\infty$ holds.}

    Now by the Skorokhod representation theorem, there exists a sequence of random variables $(\nu_t^{N_i})_{t\in [0,T]}$ and $(\nu_t)_{t\in [0,T]}$ defined on the same probability space $(\Omega, \C{F},\B{P})$, such that $(\nu_t^{N})_{t\in [0,T]}$ converges $\B{P}$-a.s. to $(\nu_t)_{t\in [0,T]}$ in the Skorokhod topology, with $(\nu_t^{N})_{t\in [0,T]}\sim Q_T^{N}$ and $(\nu_t)_{t\in [0,T]}\sim Q^\infty_T$. By Lemma \ref{asContinuity}, the map $t\mapsto\nu_t$ is continuous on $[0,T]$ $\B{P}$-a.s. Next we will prove that $\nu_s^N\Rightarrow\nu_s$ as $N\to\infty$ $\B{P}$-a.s., which is a simple consequence of continuity and convergence in the Skorokhod topology. By the triangle inequality, for $f\in C_b(\B{R}))$
    $$|\langle f,\nu_s^N\rangle-\langle f,\nu_s\rangle| \leq |\langle f,\nu_s^N\rangle- \langle f,\nu_{\lambda_N^{-1}(s)}\rangle| + |\langle f,\nu_{\lambda_N^{-1}}(s)\rangle-\langle f,\nu_s\rangle|.$$
    The first term converges to $0$ almost surely by the Skorokhod convergence of $(\nu^N_t)_{t\in [0,T]} \to (\nu_t)_{t\in [0,T]}$, and the second term converges to $0$ almost surely by the almost sure continuity of $t\mapsto\nu_t$ at $t=s$. Since $\C{BC}^{2,1}(\B{R}\times [0,\infty),\B{R})\subseteq C_b(\B{R})$, it immediately follows that for $f\in \C{BC}^{2,1}(\B{R}\times [0,\infty),\B{R})$ we have $\langle \nu_t^N, f(\cdot,t)\rangle \xrightarrow[N\to\infty]{\B{P}\text{-a.s.}}\langle \nu_t,f(\cdot,t)\rangle$ and $\langle \nu_0^N,f(\cdot,0)\rangle \xrightarrow[N\to\infty]{\B{P}\text{-a.s.}}\langle \rho,f(\cdot,0)\rangle$.
    
    Note that for $f\in \C{BC}^{2,1}(\B{R}\times [0,\infty),\B{R})$, since $r$ is bounded and continuous, thus $\frac12f_{xx}+f_t+rf$ is bounded and continuous. Thus $\langle \nu_s^N,\frac12 f_{xx}(\cdot,s)+f_t(\cdot,s)+r(s)f(\cdot,s)\rangle$ converges to $\langle \nu_s,\frac12 f_{xx}(\cdot,s)+f_t(\cdot,s)+r(s)f(\cdot,s)\rangle$ for all $s\in [0,t]$ $\B{P}$-almost surely. %Let $D\subset [0,t]$ be a countable dense subset of $[0,t]$, and let $\tilde{\Omega}\subseteq \Omega$ be such that for $\omega \in\tilde{\Omega}$, $\langle \nu_s(\omega)^N,\frac12 f_{xx}(\cdot,s)+f_t(\cdot,s)+r(s)f(\cdot,s)\rangle$ converges to $\langle \nu_s(\omega),\frac12 f_{xx}(\cdot,s)+f_t(\cdot,s)+r(s)f(\cdot,s)\rangle$ for every $s\in D$. Since $D$ is countable, and $\nu_s^N\xrightarrow[N\to\infty]{}\nu_s$ almost surely, thus $\B{P}(\tilde{\Omega})=1$. Since $\nu_s^N$ and $\nu_s$ are cadlag, $\frac12f_{xx}+f_t+rf$ is continuous in time, and $D$ is dense in $[0,t]$, thus this convergence holds for all $t\in [0,T]$. 
    Therefore by the dominated convergence theorem
    %\begin{align*}\lim_{N\to\infty}\int_0^t \langle \nu^N_s(\omega),\frac12 f_{xx}(\cdot,s)+f_t(\cdot,s)+r(s)f(\cdot,s)\rangle ds = \int_0^t \langle \nu^\infty_s(\omega),\frac12 f_{xx}(\cdot,s)+f_t(\cdot,s)+r(s)f(\cdot,s)\rangle ds\end{align*}
    %and hence
    \begin{align*}\int_0^t \langle \nu^N_s,\frac12 f_{xx}(\cdot,s)+f_t(\cdot,s)+r(s)f(\cdot,s)\rangle ds \xrightarrow[N\to\infty]{\B{P}\text{-a.s.}}\int_0^t \langle \nu^\infty_s,\frac12 f_{xx}(\cdot,s)+f_t(\cdot,s)+r(s)f(\cdot,s)\rangle ds.\end{align*}

    Extending the idea of the previous argument, since $\nu^N_s\underset{N\to\infty}{\Longrightarrow}\nu_s$ and $\nu_s\ll \lambda$ are true for all $s\in [0,T]$ $\B{P}$-a.s., define the measure $1$ set
    $$\hat{\Omega}:=\left\{\omega\in \Omega:\nu_s^N(\omega)\underset{N\to\infty}{\Longrightarrow}\nu_s(\omega),\nu_s(\omega)\ll \lambda \forall s\in [0,T]\right\}\subseteq \Omega.$$
    We will show that for $\omega\in \hat{\Omega}$, the convergence
    \begin{align}\label{branchingTermLangleConverge}\int_0^t\left\langle \nu^N_s(\omega),N f(\cdot,s)r(s)\int_{(H\star \nu_s^N(\omega))(\cdot)-1/N}^{(H\star \nu_s^N(\omega))(\cdot)}\psi(u)du\right\rangle ds\xrightarrow[N\to\infty]{} \int_0^t \langle\nu_s(\omega), f(\cdot,s)r(s)\psi((H\star \nu_s(\omega))(\cdot))\rangle ds
    \end{align} 
    holds. Fix $\epsilon > 0$. We will also fix $\omega \in \hat{\Omega}$, however we  omit to explicitly show dependence of $\omega$ for readability. Since the integrand of the left-hand side of (\ref{branchingTermLangleConverge}) is bounded by %$\|f\|_\infty\|r\|_\infty \|\psi\|_\infty \sup_{N\geq 1, s\in [0,t]}\langle \mu_s^N,1\rangle=
    $\|f\|_\infty\|r\|_\infty \|\psi\|_\infty$, therefore by the dominated convergence theorem:
    \begin{align*}\lim_{N\to\infty}\int_0^t\left\langle \nu^N_s,N f(\cdot,s)r(s)\int_{(H\star \nu_s^N)(\cdot)-1/N}^{(H\star \nu_s^N)(\cdot)}\psi(u)du\right\rangle ds = \int_0^t \lim_{N\to\infty}\left\langle \nu^N_s,Nf(\cdot,s)r(s)\int_{(H\star \nu_s^N)(\cdot)-1/N}^{(H\star \nu_s^N)(\cdot)}\psi(u)du\right\rangle ds.\end{align*}
    By the triangle inequality
    \begin{align}\nonumber\Big|\Big\langle \nu^N_s,Nf(\cdot,s)r(s)&\int_{(H\star \nu_s^N)(\cdot)-1/N}^{(H\star \nu_s^N)(\cdot)}\psi(u)du\Big\rangle - \langle \nu_s, f(\cdot,s)r(s)\psi((H\star \nu_s)(\cdot))\rangle\Big| \\
    &\label{triangleIneqLangles1}\leq \Big|\Big\langle \nu^N_s,N f(\cdot,s)r(s)\int_{(H\star \nu_s^N)(\cdot)-1/N}^{(H\star \nu_s^N)(\cdot)N}\psi(u)du\Big\rangle - \langle \nu_s^N, f(\cdot,s)r(s)\psi((H\star \nu_s^N)(\cdot))\rangle\Big| \\
    &\label{triangleIneqLangles2} + |\langle \nu_s^N, f(\cdot,s)r(s)\psi((H\star \nu_s^N)(\cdot))\rangle-\langle \nu_s^N, f(\cdot,s)r(s)\psi((H\star \nu_s)(\cdot))\rangle|\\
    &\label{triangleIneqLangles3}+ |\langle \nu_s^N, f(\cdot,s)r(s)\psi((H\star \nu_s)(\cdot))\rangle-\langle \nu_s, f(\cdot,s)r(s)\psi((H\star \nu_s)(\cdot))\rangle|
    \end{align}
    Since $\psi$ is continuous and $[0,1]$ is compact, there exists $N_0$ such that $N>N_0$ implies that for all $x\in [0,1]$, we have $|\psi(x)-\psi(u)|<\epsilon/(4\|f\|_\infty \|r\|_\infty)$ for all $u\in [x-1/N,x]$. Therefore for $N>N_0$:
    \begin{align*}\Big|\Big\langle \nu^N_s,&Nf(\cdot,s)r(s)\int_{(H\star \nu_s^N)(\cdot)-1/N}^{(H\star \nu_s^N)(\cdot)}\psi(u)du\Big\rangle - \langle \nu_s^N, f(\cdot,s)r(s)\psi((H\star \nu_s^N)(\cdot))\rangle\Big|
    \\ &\leq \Bigg\langle \nu_s^N, \|f\|_\infty\|r\|_\infty\Bigg|N\int_{(H\star\nu_s^N)(\cdot)-1/N}^{(H\star \nu_s^N)(\cdot)}\psi(u)du-\psi((H\star\nu_s^N)(\cdot)\Bigg|\Bigg\rangle \leq \Big|\Big\langle \nu_s^N, \frac{\epsilon}{4}\Big\rangle\Big| \leq \epsilon/4.\end{align*}

    This proves that \eqref{triangleIneqLangles1} converges to $0$ as $N\to\infty$. Next we tackle \eqref{triangleIneqLangles2}. Since $\nu_s^N$ converges weakly to $\nu_s$ and $\nu_s \ll \lambda$, thus $(H\star \nu_s^N)(x)\to(H\star \nu_s)(x)$ for all $x\in \B{R}$ as $N\to\infty$. By continuity of $\psi$, $\psi((H\star \nu_s^N)(x))$ converges to $\psi((H\star \nu_s)(x))$ for all $x\in \B{R}$. Moreover, since $\nu_t^N\underset{N\to\infty}{\Longrightarrow}\nu_t$, therefore by Prokhorov's theorem (Theorem 5.1, \cite{billingsley}), $(\nu_t^N)_{N=1,2,\ldots}$, is tight, therefore we can choose a compact set $K_\epsilon$ such that $\nu_s^N(\B{R}\setminus K_\epsilon)<\epsilon/(4\|f\|_\infty \|r\|_\infty \|\psi\|_\infty)$ for all $N\in \B{N}$. So
    \begin{align*}
        \Bigg| \int_{\B{R}\setminus K_\epsilon} &f(x,s)r(s)\psi((H\star \nu_s^N)(x))\nu_s^N(dx)-\int_{\B{R}\setminus K_\epsilon} f(x,s)r(s)\psi((H\star \nu_s)(x))\nu_s^N(dx)\Bigg|\leq \epsilon/4
    \end{align*}
    Then since $K_\epsilon$ is compact, $\psi((H\star \nu_s^N)(x))$ converges to $\psi((H\star \nu_s^\infty)(x))$ uniformly on $K_\epsilon$ as $N\to\infty$. Then since $f(\cdot,s)$ and $r$ are bounded, there exists $N_1$ such that $N>N_1$ implies
    $$|f(x,s)r(s)\psi((H\star \nu_s^N)(x))-f(x,s)r(s)\psi((H\star \nu_s^\infty)(x))|<\epsilon/4,$$
    for all $x\in K_\epsilon$, and therefore for any probability measure $\theta$, we have that $N>N_1$ implies
    \begin{align*}
        \Bigg| \int_{K_\epsilon} &f(x,s)r(s)\psi((H\star \nu_s^N)(x))\theta(dx)-\int_{K_\epsilon} f(x,s)r(s)\psi((H\star \nu_s)(x))\theta(dx)\Bigg|\\
        &\leq \int_{K_\epsilon} |f(x,s)r(s)\psi((H\star \nu_s^N)(x))-f(x,s)r(s)(H\star \nu_s)(x) |\theta(dx)\leq \frac{\epsilon}{4}\int_{K_\epsilon}\theta(dx)\leq \epsilon/4.
    \end{align*}
    Therefore, by the triangle inequality, for $N>N_1$, we have that
    \begin{align*}
        |\langle \nu_s^N, f(\cdot,s)r(s)\psi((H\star \nu_s^N)(\cdot))\rangle - \langle \nu_s^N, f(\cdot,s)r(s)\psi((H\star \nu_s)(\cdot))\rangle| \leq \epsilon/2
    \end{align*}
    
    This proves convergence of \eqref{triangleIneqLangles2}. Finally we tackle \eqref{triangleIneqLangles3}. Observe that $f(x,s)r(s)\psi((H\star \nu_s)(x))$ is continuous and bounded in $x$, therefore by the Portmanteau theorem, since $\nu_s^N \Rightarrow \nu_s$, we have
    $$\langle \nu_s^N, f(\cdot,s)r(s)\psi((H\star \nu_s)(\cdot))\rangle \xrightarrow[N\to\infty]{}\langle \nu_s, f(\cdot,s)r(s)\psi((H\star \nu_s)(\cdot))\rangle.$$
    So let $N_2$ be such that $N>N_2$ implies that $$|\langle \nu_s^N, f(\cdot,s)r(s)\psi((H\star \nu_s)(\cdot))\rangle -\langle \nu_s, f(\cdot,s)r(s)\psi((H\star \nu_s)(\cdot))\rangle|<\epsilon/4,$$ 
    thus for $N>\max\{N_0,N_1,N_2\}$, we have that:
    $$\Big|\Big\langle \nu^N_s,Nf(\cdot,s)r(s)\int_{(H\star \nu_s^N)(\cdot)-1/N}^{(H\star \nu_s^N)(\cdot)}\psi(u)du\Big\rangle - \langle \nu_s, f(\cdot,s)r(s)\psi((H\star \nu_s)(\cdot))\rangle\Big|\leq \epsilon.$$
    %This holds for all $s\in D$ and $\omega\in \hat{\Omega}$. Since a right-continuous function is determined by its values on a dense set, we therefore have point-wise convergence everywhere. 
    Therefore \eqref{branchingTermLangleConverge} holds for all $\omega\in \hat{\Omega}$, hence $\B{P}$-almost surely. Therefore for any suitable test function $f\in \C{BC}^{2,1}(\B{R}\times [0,\infty),\B{R})$, we have 
    \begin{align}\label{convOfGmuN}G(\mu^N,f,N,t)\xrightarrow[N\to\infty]{\B{P}\text{-a.s.}} \langle \mu_t^\infty&, f(\cdot,t)\rangle - \langle \rho, f(\cdot,0)\rangle \\
    \nonumber &- \int_0^t \langle \mu_s^\infty, \frac12 f_{xx}(\cdot,s) + f_t(\cdot,s) + r(s)f(\cdot,s)\rangle + \langle \mu_s^\infty,f(\cdot,s)r(s)\psi((H\star \mu_s^\infty)(\cdot))\rangle ds.\end{align}

    Fixing $t$ and $f$, and recalling the boundedness of each term of $G(\mu^N,f,N,t)$, we can observe that the sequence $(G(\mu^N,f,N,t))_{N\geq 0}$ is bounded uniformly in $N$. Hence $(|G(\mu^N,f,N,t)|)_{N\geq 0}$ is uniformly integrable. Then since convergence in expectation and uniform integrability imply almost sure convergence (see for example Appendixes, Proposition 2.3 in \cite{ethierKurtz}).%, thus and equation (\ref{EGto0}), we have that
    %$$\B{E}[\lim_{N\to\infty}|G(\mu^N,f,N,t)|^2]=\lim_{N\to\infty}\B{E}[|G(\mu^N,f,N,t)|^2]=0.$$
    Therefore \eqref{EGto0} yields that $\lim_{N\to\infty}G(\mu^N,f,N,t)=0$ $\B{P}$-a.s., completing the proof. 
\end{proof}

%The preceding results therefore show that $(\mu^N)_{N\geq 1}$ has subsequential limits, that any subsequential limit $\mu^\infty$ of the sequence $(\mu^N)_{N\geq 1}$ is almost surely a solution of the equation (\ref{muinfLimit}). 
Next, we show the deterministic result that any solution of \eqref{muinfLimit} with $\rho\ll \lambda$ has $\mu_s\ll \lambda$ for all $s>0$.

\vspace{1em}

\begin{propn} \label{absContEverywhere}
    Let $(\nu_t)_{t\geq 0}$ satisfy equation \eqref{muinfLimit} with $\nu_0=\rho\ll \lambda$. Then $\nu_t \ll \lambda$ for any $t>0$. 
\end{propn}

\begin{proof}
    We will prove that $\nu_t \ll \lambda$ by showing that $\exists C=C(t)>0$ such that for any interval $[a,b]\subset \B{R}$, we have $\nu_t([a,b])\leq C|b-a|$, since this shows that $G(x):=\nu_t((-\infty,x])$ is Lipschitz continuous, a strictly stronger condition than absolute continuity. So fix $t$ and $[a,b]\subset \B{R}$. Let $I(x):=\is_{\{x\in [a,b]\}}$ and $p_t(x):=\frac{1}{\sqrt{2\pi} t}\exp(-x^2/2t)$ be the heat kernel. Since $\rho\ll\lambda$, say $\rho$ has density $\xi$ (i.e. $\rho(dx)=\xi(x)dx$), and let $m^n$ be a sequence of mollifiers such that $(m^n\star \xi)(x)$ is a monotonic increasing sequence converging pointwise to $\xi(x)$ as $n\to\infty$. 
    
    We define as a sequence of test functions $\phi^n(x,s):=(p_{t-s}\star m^n \star I)(x)$. This is the solution of the backward heat equation on $\B{R}\times [0,t]$ with terminal condition $(m^n\star I)(x)$; that is $\frac12\phi^n_{xx}+\phi_t=0$ and $\phi^n(x,t)=(m^n\star I)(x)$. Plugging this into \eqref{muinfLimit}, we yield:
    \begin{align}\label{mollifiedTestRep}
        \langle \nu_t,\phi^n(\cdot,t)\rangle = \langle \rho, \phi^n(\cdot,0)\rangle + \int_0^t \langle \nu_s, r(s)\phi^n(\cdot,s)(1-\psi((H\star \nu_s)(\cdot))\rangle ds
    \end{align}
    Now since $\phi^n(x,s)$ is a monotonically increasing sequence converging to $(p_{t-s}\star I)(x)$ thus %by the monotone convergence theorem (MCT)
    %\begin{align*} \langle \nu_t, \phi^n(\cdot,t)\rangle \xrightarrow[n\to\infty]{} \langle \nu_t, I\rangle =  \nu_t([a,b])\end{align*}
    using the monotone convergence theorem and Fubini's theorem
    \begin{align*} \langle \rho, \phi^n(\cdot,0)\rangle \xrightarrow[n\to\infty]{} \langle \rho, (p_t \star I)(\cdot)\rangle &= \int_{\B{R}} \left(\int_{\B{R}} p_t(x-y)I(y)dy\right) \xi(x)dx = \int_{[a,b]} (p_t\star \xi)(y)dy \\
    &\leq \|(p_t\star \xi)\|_\infty |b-a|\end{align*}
    As $\psi$ is positive and $r$ bounded, we can bound $\langle \nu_s, r(s)\phi^n(\cdot,s)(1-\psi((H\star \nu_s)(\cdot))\rangle \leq \|r\|_\infty \langle \nu_s, \phi^n(\cdot,s)\rangle$, and hence taking limits $n\to\infty$ in equation \eqref{mollifiedTestRep} and applying the monotone convergence theorem again to the integral term, we yield
    \begin{align*}
        \langle \nu_t, p_0 \star I \rangle \leq \|(p_t \star \xi)\|_\infty |b-a| + \int_0^t \|r\|_\infty \langle \nu_s, p_{t-s}\star I\rangle ds,
    \end{align*}
    so that Gr\"{o}nwall's inequality gives $\nu_t([a,b])=\langle \nu_t,p_0\star I\rangle \leq \|(p_t\star \xi)\|_\infty |b-a| e^{\|r\|_\infty t}$. This holds for any interval $[a,b]\subset \B{R}$, hence $\nu_t\ll \lambda$. 
\end{proof}

Next we show that if $(\mu^\infty_t)_{t\geq 0}$ satisfies (\ref{muinfLimit}), then $F(x,t):=\int_x^\infty \mu^\infty_t(dy)$ is the weak solution of the PDE (\ref{mainResultPDE}). As we do not yet know that the function $F(x,t)$ is differentiable either in $x$ or $t$ we cannot yet say that it is a classical solution. Consider taking
%\commentout{It is natural in this case to consider the weak solution of the PDE (\ref{mainResultPDE}) as it is not yet known that the function $F(x,t)$ is differentiable either in $x$ or $t$.} Take 
$U$ a classical solution of \eqref{mainResultPDE}, 
%If a solution $U$ to the PDE (\ref{mainResultPDE}) was to be differentiable once in time and twice in space, 
then multiplying by a test function $f\in C_c^{2,1}(\B{R}\times [0,\infty),\B{R})$ and integrating over $\B{R}\times [0,t]$  yields:
\begin{align}\label{weakSolnExample} 0 = \int_{\B{R}\times [0,t]} \left(U_t(x,t) -\frac12 U_{xx}(x,t) - r(t)U(x,t)+r(t)\int_{1-U(x,t)}^1 \psi(s)ds\right)f(x,t)dxdt,\end{align}
and then switching the derivatives from $U$ to $f$ gives 
\begin{align*}\int_{\B{R}\times [0,t]}U_t(x,s)f(x,s)(dx,ds)&= \int_{\B{R}}U(x,t)f(x,t)dx -\int_{\B{R}}U(x,0)f(x,0)dx - \int_{\B{R}\times [0,t]}U(x,t)f_t(x,t)dx\,dt\\
&=\langle U(\cdot,t),f(\cdot,t)\rangle - \langle U(\cdot,0),f(\cdot,0)\rangle - \int_0^t \langle U(\cdot,s),f_t(\cdot,s)\rangle ds\end{align*}
and 
$$\int_{\B{R}\times [0,t]}U_{xx}(x,t)f(x,t)(dx,ds)=\int_{\B{R}\times [0,t]} U(x,t)f_{xx}(x,t)(dx,ds)=\int_0^t \langle U(\cdot,s),f_{xx}(\cdot,s)\rangle ds$$
therefore equation (\ref{weakSolnExample}) could be written:
\begin{align} \label{UweakSolution}
        \langle U(\cdot,t),f(\cdot,t)\rangle - \langle U(\cdot,0), f(\cdot,0)\rangle = \int_0^t \langle U(\cdot,s),\frac12 f_{xx}(\cdot,s) &+ f_t(\cdot,s) + r(s)f(\cdot,s)\rangle \\
        &- r(s)\left\langle f(\cdot,s),\int_{1-U(\cdot,s)}^1 \psi(u)du\right\rangle ds. \notag
\end{align} 
Notice that no assumptions on the differentiability of $U$ are required in equation \eqref{UweakSolution}. This is the natural definition for a weak solution of equation (\ref{mainResultPDE}). 

\vspace{1em}

\begin{defn}
    We call $U:\B{R}\times [0,t]\to \B{R}$ a weak solution of the PDE (\ref{mainResultPDE}) if (\ref{UweakSolution}) holds for all test functions $f$ in a dense subset of $C_c^{2,1}(\B{R}\times [0,t],\B{R})$.
\end{defn}

\vspace{1em}

\begin{propn} \label{cdfPdfEquiv}
    Suppose that $(\mu_t)_{t\geq 0}$ satisfies (\ref{muinfLimit}); that is
    \begin{align*} \langle \mu_t, f(\cdot,t)\rangle = \langle \rho ,f(\cdot,0)\rangle + \int_0^t \langle \mu_s,\frac12 f_{xx}(\cdot,s)+f_t(\cdot,s) + r(s)f(\cdot,s)\rangle - \left\langle \mu_s,f(\cdot,s)r(s)\psi((H\star\mu_s)(\cdot)\right\rangle ds, 
    \end{align*}
    for every $f\in \C{BC}^{2,1}(\B{R}\times [0,\infty),\B{R})$. %, that $\mu_0(\B{R})=1$, and that $\mu_0\ll \lambda$. 
    Define $F(x,t):=(\tilde{H}\star \mu_t)(x)$. Then $\langle \mu_t,1\rangle=1$ for all $t\geq 0$ and $F$ %is a weak solution to the PDE (\ref{mainResultPDE}) in the sense that $F$ 
    satisfies
    \begin{align}\label{cdfFormulation}
        \langle F(\cdot,t),f(\cdot,t)\rangle - \langle \tilde{H}\star \rho, f(\cdot,0)\rangle = \int_0^t \langle F(\cdot,s),\frac12 f_{xx}(\cdot,s) &+ f_t(\cdot,s) + r(s)f(\cdot,s)\rangle \\
        &- r(s)\left\langle f(\cdot,s),\int_{1-F(\cdot,s)}^1 \psi(u)du\right\rangle ds. \notag
    \end{align}
    for all test functions $f$ such that $(H\star f(\cdot,t))(x)\in \C{BC}^{2,1}(\B{R}\times [0,\infty),\B{R})$.
\end{propn}

\begin{proof}
    \sloppy Firstly we show that any solution $\mu_t$ of (\ref{muinfLimit}) such that $\mu_t\ll\lambda$ and $\rho(\B{R})=1$ has $\langle \mu_t,1\rangle = 1$ for all $t\geq 0$. Substituting $f(x,s)\equiv 1\in \C{BC}^{2,1}(\B{R}\times [0,\infty),\B{R})$ into (\ref{muinfLimit}) and differentiating, we get that
    \begin{align*}\frac{d}{dt}\langle \mu_t, 1\rangle = r(t)\langle \mu_t, 1\rangle - r(t)\langle \mu_t, \psi((H\star \mu)(\cdot))\rangle.\end{align*} 
    Since $\mu_t$ is absolutely continuous with respect to the Lebesgue measure, therefore its cumulative distribution function $\tilde{F}(x,t):=(H\star \mu_t)(x)$ is continuous in $x$, and $\tilde{F}(x,t):\B{R}\to [0,1]$ is a measurable function and has a well-defined inverse $\tilde{F}^{-1}(x,t)$ with respect to $x$. Define by $\mu_s \circ \tilde{F}^{-1}(\cdot,s):=\mu_s(\tilde{F}^{-1}(\cdot,s))$ the pushforward measure of $\mu_s$ by $\tilde{F}(\cdot,s)$, which is the Lebesgue measure on $[0,1]$ (see Example 3.6.2 of \cite{bogachev}, for example). So by Theorem 3.6.1 of \cite{bogachev}, we have that:
    \begin{align*}
        \langle \mu_t, \psi((H\star \mu_t)(\cdot))\rangle = \int_{\B{R}} \psi((H\star \mu_t)(x))\mu_t(dx) = \int_{\tilde{F}(-\infty,t)}^{\tilde{F}(\infty,t)}\psi(u)(\mu_t\circ \tilde{F}^{-1}(\cdot,t))(du) = \int_0^1 \psi(u)du = 1
    \end{align*}

    Therefore $\langle \mu_t,1\rangle$ solves the ordinary differential equation $\frac{d}{dt}\langle \mu_t,1\rangle = r(t)\langle \mu_t,1\rangle - r(t)$ with $\langle \mu_0,1\rangle=1$, and therefore $\langle \mu_t,1\rangle \equiv 1$, as desired. Now recall that we have the relations $\langle \tilde{H}\star a,b\rangle =\langle a,H\star b\rangle$, $\langle H\star a, b\rangle = \langle a,\tilde{H}\star b\rangle$, and $\frac{\partial}{\partial x}(a\star b)(x)=(a'\star b)(x) = (a\star b')(x)$. So by Fubini's theorem, we have
    \begin{align*}
        \langle \mu_s, (H\star f(\cdot,s))(\cdot)r(s)\psi((H\star \mu_s)(\cdot))\rangle &= r(s)\int_{\B{R}}(H\star f(\cdot,s))(x)\psi((H\star \mu_s)(x))\mu_s(dx) \\ &=r(s)\int_{\B{R}}\int_{\B{R}}\is_{x-y\geq 0}f(y,s)dy\,\psi((H\star \mu_s)(x))\mu_s(dx) \\
        &= r(s)\int_{\B{R}}f(y,s)\int_{y}^{\infty}\psi(H\star \mu_s)(x)\mu_s(dx)dy 
    \end{align*}
    Again, using Theorem 3.6.1 of \cite{bogachev}, we have that:
    $$\int_y^\infty \psi((H\star \mu_s)(x))\mu_s(dx)=\int_{y}^\infty \psi(\tilde{F}(x,s))\mu_s(dx) = \int_{\tilde{F}(y,s)}^{\tilde{F}(\infty,s)}\psi(u)(\mu_s \circ \tilde{F}^{-1}(\cdot,s))(du)=\int_{\tilde{F}(y,s)}^1 \psi(u)du.$$
    Therefore
    $$\langle \mu_s, (H\star f(\cdot,s))(\cdot)r(s)\psi((H\star \mu_s)(\cdot))\rangle = r(s)\int_{\B{R}}f(y,s)\int_{\tilde{F}(y,s)}^1 \psi(u)du\,dy = r(s) \left\langle f(\cdot,s), \int_{1-F(\cdot,s)}^1 \psi(u)du\right\rangle,$$
    and hence for $f(x,t)$ such that $(H\star f(\cdot,t))(x)\in \C{BC}^{2,1}(\B{R}\times [0,\infty),\B{R})$:
    \begin{align*}
        \langle F(\cdot,t),f(\cdot,t)\rangle = \langle \tilde{H}\star &\mu_t,f(\cdot,t)\rangle = \langle \mu_t , H\star f(\cdot,t)\rangle \\
        = \langle \rho,H\star f(\cdot,0)&\rangle + \int_0^t \langle \mu_s, \frac12 \frac{\partial^2 (H\star f(\cdot,s))}{\partial x^2}+\frac{\partial (H\star f(\cdot,s))}{\partial s} + r(s)(H\star f(\cdot,s))\rangle \\ &- \langle \mu_s,(H\star f(\cdot,s))(\cdot)r(s)\psi(H\star\mu_s)(\cdot))\rangle ds \\
        = \langle\tilde{H}\star \rho,f(\cdot,0)\rangle & + \int_0^t \langle F(\cdot,s), \frac12 f_{xx}(\cdot,s)+f_t(\cdot,s) + r(s)f(\cdot,s)\rangle - r(s)\left\langle f(\cdot,s),\int_{1-F(\cdot,s)}^1 \psi(u)du\right\rangle ds.
    \end{align*}\end{proof}

Note that the functions $f$ with compact support such that $(H\star f(\cdot,t))(x)\in \C{BC}^{2,1}(\B{R}\times [0,\infty),\B{R})$ form a dense subset of $C^{2,1}_c(\B{R}\times [0,\infty),\B{R})$, so the above proposition shows that $F$ is in fact a weak solution of the PDE \eqref{mainResultPDE}. Next we show that the weak formulation (\ref{cdfFormulation}) has a unique solution.

\vspace{1em}

\begin{propn} \label{Uniqueness}
    Fix $T>0$ and $\rho\ll\lambda$. Then there is at most one function $F(x,t):\B{R}\times [0,T]\to\B{R}$ solving \eqref{cdfFormulation} for all $f$ such that $(H\star f(\cdot,t))(x)\in\C{BC}^{2,1}(\B{R}\times [0,\infty),\B{R})$, with $F(x,0)=\int_{x}^\infty \rho(dx)$, and such that $F(x,t)$ is continuous in $x$ for all $t\in [0,T]$
\end{propn}

\begin{proof}
    Suppose that $F_1,F_2:\B{R}\times [0,\infty)\to\B{R}$ both solve \eqref{cdfFormulation} and are both continuous in $x$ for all $t\in [0,T]$. Define their difference $D_t(x)=F_1(x,t)-F_2(x,t)$ which is also continuous in $x$. Now fix $t\geq 0$ and let $\varphi$ be a test function which is positive and smooth with compact support and total mass $\int_{\B{R}}\varphi(y)dy=1$. Then let $f$ be the solution of the backward heat equation $f_t+\frac12f_{xx}+rf=0$ with terminal condition $f(x,t)=\varphi(x)$. Therefore \eqref{cdfFormulation} yields
    %\begin{align*}\langle F_1(\cdot,t), \varphi\rangle = \langle \tilde{H}\star \rho, f(\cdot,0)\rangle - \int_0^t r(s)\Big\langle \int_{1-F_1(\cdot,s)}^1 \psi(u)du,f(\cdot,s)\Big\rangle ds,\end{align*}
    %and 
    %\begin{align*}\langle F_2(\cdot,t), \varphi\rangle = \langle \tilde{H}\star \rho, f(\cdot,0)\rangle - \int_0^t r(s)\Big\langle \int_{1-F_2(\cdot,s)}^1 \psi(u)du,f(\cdot,s)\Big\rangle ds.\end{align*}
    %Taking the difference of these two equations yields
    \begin{align} \label{differenceForUniquenessProof}
        \langle D_t,\varphi\rangle = \int_0^t r(s)\Big\langle \int_{1-F_2(\cdot,s)}^1 \psi(u)du - \int_{1-F_1(\cdot,s)}^1 \psi(u)du, f(\cdot,s)\Big\rangle ds.
    \end{align}
    Then as $\psi$ is bounded,
    $$\left|\int_{1-F_2(x,s)}^1 \psi(u)du-\int_{1-F_1(x,s)}^1\psi(u)du\right|\leq \|\psi\|_\infty|F_2(x,s)-F_1(x,s)|\leq \|\psi\|_\infty|D_s(x)|$$
    Therefore as $r$ is bounded and $\int_{\B{R}}f(x,s)dx\leq e^{\|r\|_\infty T}$ for all $s\in [0,T]$, therefore \eqref{differenceForUniquenessProof} gives that
    \begin{align*}
        |\langle D_t, \varphi\rangle | \leq \|\psi\|_\infty \|r\|_\infty \int_0^t \int_{\B{R}} |D_s(x)f(x,s)|dx\,ds \leq \|\psi\|_\infty \|r\|_\infty e^{\|r\|_\infty T}\int_0^t \|D_s\|_\infty ds.
    \end{align*}
    Now let us consider the supremum of $|\langle D_t,\varphi\rangle|$ over all smooth positive functions $\varphi$ with total mass $1$. Clearly $|\int_{\B{R}} D_t(x)\varphi(x)dx|\leq \|D_t\|_\infty \int_{\B{R}}\varphi(x)dx=\|D_t\|_\infty$, so the supremum is clearly bounded by $\|D_t\|_\infty$. On the other hand, since $D_t$ is continuous, choosing $\varphi$ as a smooth approximation to $\delta(\cdot-x^*)$, where $x^*$ is the point at which $D$'s supremum is achieved, we can clearly choose $\varphi$ such that $|\langle D_t,\varphi\rangle|$ is arbitrarily close to $\|D_t\|_\infty$. Thus $\langle D_t,\varphi\rangle = \|D_t\|_\infty$. Therefore we have:
    $$\|D_t\|_\infty \leq \|\psi\|_\infty \|r\|_\infty e^{\|r\|_\infty T}\int_0^t \|D_s\|_\infty ds,$$
    and so by Gr\"{o}nwall's inequality, it follows that $\|D_t\|_\infty\equiv 0$ for $t\in [0,T]$, and so $D_t(x)\equiv 0$. Thus \eqref{cdfFormulation} has a unique solution which is continuous in $x$.
\end{proof}

\begin{proof}(Of Theorem \ref{psirNGeneralLimit})
    By tightness (Proposition \ref{tightness}) the sequence of laws $(Q^N_T)_{N=1,2,\ldots}$ has subsequential limits. By Proposition \ref{detConvPropn}, if $Q^\infty_T$ is a subsequential limit and $\mu^\infty\sim Q^\infty_T$, then $(\mu_t^\infty)_{t\in [0,T]}$ is a weak solution to the PDE \eqref{muinfLimit}. Moreover, since $\rho=\mu^\infty_0\ll \lambda$, Proposition \ref{absContEverywhere} tells us that $\mu_t^\infty\ll \lambda$ for all $t\in [0,T]$. Therefore by Propositions \ref{cdfPdfEquiv} and \ref{Uniqueness}, $(\tilde{H}\star \mu_t^\infty)(x)$ is the unique solution to \eqref{cdfFormulation}; that is, the unique weak solution to \eqref{mainResultPDE}. Then as \eqref{mainResultPDE} has a unique classical solution (\S3, Theorem 1, \cite{kpp}) this coincides with the unique weak solution. Since, therefore, all subsequential limits of $(Q^N_T)_{N=1,2,\ldots}$ are the same (i.e. a Dirac mass at $u=-\frac{\partial}{\partial x}U(x,t)$, where $U$ is the unique classical solution to \eqref{mainResultPDE}), thus the sequence $(Q^N_T)_{N=1,2,\ldots}$ converges weakly to this limit. 
\end{proof}

Let us conclude this section with a remark about the necessity of the assumption \textit{\textbf{(A1)}}. Note that he $N$-BBM can be described as $(\psi,r,N)$-BBM process where $r(t)\equiv 1$ and $\psi(x)=\delta_0(x)$, since at each branching time, we select the leftmost particle to be deleted with probability $\int_0^{1/N}\delta_0(dx)=1$. Such a $\psi$ certainly fails the assumption, since $\psi$ is not even a function. And indeed, the boundedness of $\psi$ is crucially utilised in the proofs of Propositions \ref{weakRepPropn}, \ref{tightness}, and Lemmas \ref{absContEveryTime}, \ref{asContinuity}, and continuity is needed in the proof of Prposition \ref{detConvPropn}. Unfortunately, therefore, it is non-trivial to see how the methods here would, without non-trivial extension, reproduce the hydrodynamic limit result of \cite{hydroNBBM}.

\section{General reaction-diffusion equations} \label{generalFKPPSection}

So far we have shown that the $(\psi,r,N)$-BBM process has a hydrodynamic limit described by the unique solution to the reaction-diffusion equation
$$U_t = \frac12 U_{xx} + r(t)\left( U-\int_{1-U}^1 \psi(s)ds\right).$$

In particular, this means that given any function $G(x)$ with $G(0)=G(1)=0$ and $G'(x)\leq 1$, we can describe the solution of the PDE $U_t = \frac12 U_{xx} + G(U)$ as the hydrodynamic limit of a branching-selection particle system with selection function defined by $\psi(x)=1-G'(1-x)$. The author hopes that this introduces an intuitive way in which to understand a number of general reaction-diffusion equations, since $\psi$ can be thought of as describing the relative probability of an individual in the population being out-competed as a function of ranked fitness.

\paragraph{Example 1:} A ubiquitous choice of $G$ is the Fisher equation, when $G(x)=x(1-x)$, which has been studied since it was introduced by Fisher \cite{fisher}. Choosing $\psi(x)=2-2x$ gives us a particle system who's hydrodynamic limit is the equation $U_t = \frac12 U_{xx} + U - U^2$. Simply, this particle system is the system of $N$ particles, moving as Brownian motions, in which each particle branches at unit rate and simultaneously a particle is deleted with a probability which is directly proportional to its rank. %It has already been shown by Groisman, Jonckheere, and Martinez \cite{fkppScalingLimit} that solutions of this equation can also be understood as the hydrodynamic limit of a branching-selection particle system described as follows: $N$ particles move on $\B{R}$ as Brownian motions, and at constant rate, draw two particles from the population and move the leftmost particle to the location of the rightmost. 

\paragraph{Example 2:} The Allen-Cahn equation $U_t=\frac12 U_{xx}+U(1-U)(U-\theta)$, for $\theta\in (0,1)$, models the \textit{strong Allee effect}. In this case, to realise the PDE as the hydrodynamic limit of a rank-dependent branching selection system, we choose $\psi(x)=(2-\theta)(1-2x)+3x^2$

\paragraph{Example 3:} Another common example we give is $G(x)=x(1-x)(1+\rho x)$, with $\rho\in [0,1]$. In this case, we can realise $U_t = \frac12 U_{xx} + U(1-U)(1+\rho U)$ as the hydrodynamic limit of the rank-dependent branching selection system with selection function $\psi(x)=2+a-2x+4ax+3ax^2$.

\begin{figure}[h]
    \centering
    \begin{subfigure}{0.32\textwidth}
        \includegraphics[width=\linewidth]{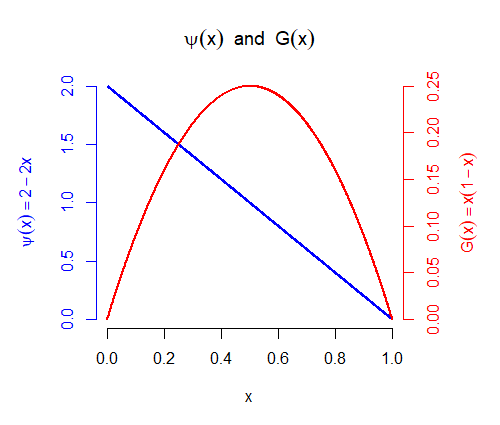}
        %\caption{First plot}
    \end{subfigure}
    \begin{subfigure}{0.32\textwidth}
        \includegraphics[width=\linewidth]{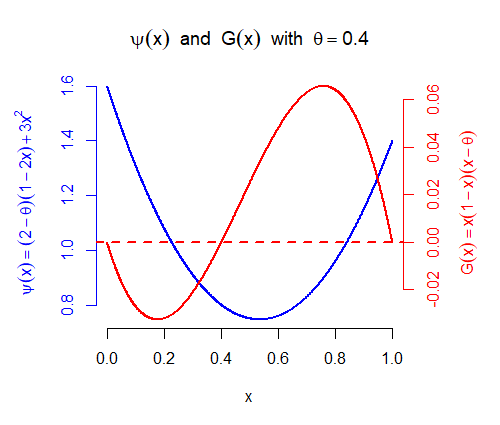}
        %\caption{Second plot}
    \end{subfigure}
    \begin{subfigure}{0.32\textwidth}
        \includegraphics[width=\linewidth]{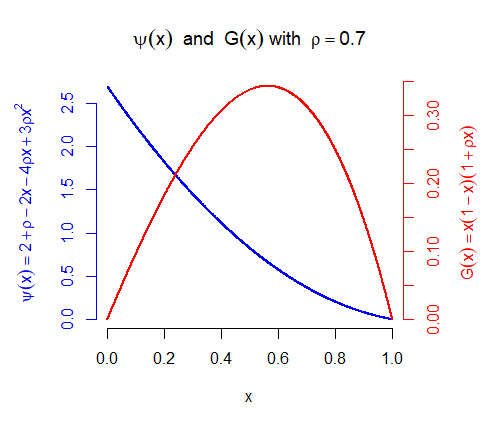}
        %\caption{Third plot}
    \end{subfigure}
    \caption{Functions $G$ and the corresponding selection functions $\psi$ for Examples 1, 2, \& 3.}
\end{figure}

\paragraph{Example 4:} Let us consider another way in which the PDE connection may be insightful. Consider a branching-selection particle system with rank-dependent selection such that particles of intermediate rank are deleted, whereas particles of extreme ranks (ie. close to $0$ or $N$) are never deleted. What should we expect the behaviour of such a particle system to be?

    Since the `middle' particles are deleted but both extremes do not get deleted, we expect that the system will split into two clouds of particles, one travelling to $-\infty$ and one to $+\infty$. But what proportion of particles is in each cloud? For a concrete example, consider $\psi(x)=\max\{0,K(x-0.8)(0.3-x)\}$, where $K$ is a scaling constant so that $\int_0^1 \psi(s)ds=1$. So the leftmost $30\%$ and rightmost $20\%$ of particles are never killed. We should expect there to be more particles in the left-moving cloud than in the right-moving cloud, but by how many? 

    Consider the corresponding reaction-diffusion equation with $G(x)=x-\int_{1-x}^1 \psi(s)ds$. This $G$ is a \textit{monostable} source term, with a single stable zero at the value $\alpha$ such that $\alpha = \int_{1-\alpha}^1 \psi(s)ds$. For this example, $\alpha\approx 0.425$. Therefore for any initial condition $U_0$ such that $\lim_{x\to-\infty}U_0(x)=1$ and $\lim_{x\to\infty}U_0(x)=0$, the solution to the PDE (\ref{mainResultPDE}) tends everywhere to $\alpha$ as $t\to\infty$, leaving us with $1-\alpha$ mass at $-\infty$ and $\alpha$ mass at $+\infty$. 

    \begin{figure}[h]
    \centering
        \includegraphics[width=0.6\linewidth]{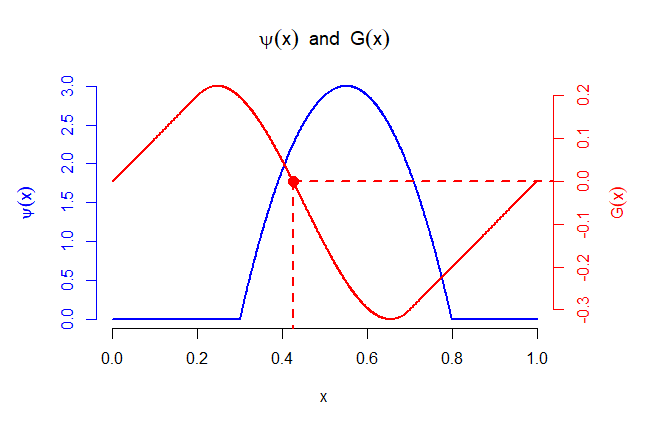}
        \caption{Selection function $\psi(x)=\max\{0,K(x-0.8)(0.3-x)\}$ and  corresponding $G(x)$.}
    \end{figure}

    Thus we expect roughly $42.5\%$ of the particles to converge to $\infty$ and $57.5\%$ of the particles to converge to $-\infty$. 
    
\section{Weak selection principle} \label{travellingSection}

In this section we prove Theorems \ref{asympSpeedTheorem} and \ref{partialWeakSelecTheorem}. First we prove Theorem \ref{asympSpeedTheorem} by constructing couplings between the $(\psi,1,N)$-BBM and the $N$-BBM process, whose asymptotic velocity is known (see Appendix of \cite{me1}). The idea for this proof is due to Paul de Lambert des Granges and produced here with his permission. Recall that we are trying to prove the following:

\vspace{1em}

\begin{theorem*}
    Let $(X^N(t))_{t\geq 0}$ be a $(\psi,r,N)$-BBM with any initial condition. Then under assumption \textit{(A2)}, the $(\psi,r,N)$-BBM has asymptotic velocity:
    $$\lim_{t\to\infty}\frac{\Theta^N_1(X^N(t))}{t}=\lim_{t\to\infty}\frac{\Theta^N_N(X^N(t))}{t} = v_N^\psi = \sqrt{2}-\frac{\pi^2}{\sqrt{2}(\log N)^2} + o\left(\frac{1}{(\log N)^2}\right).$$ 
\end{theorem*}

\begin{proof}
    We will prove the above proposition by coupling the $(\psi,1,N)$-BBM above and below to an $N$-BBM process. It is known (Theorem 21, \cite{me1}) that the $N$-BBM has asymptotic velocity $v_N = \sqrt{2}-\frac{\pi^2}{\sqrt{2}(\log N)^2} + o\Big(\frac{1}{(\log N)^2}\Big)$.

    Now let $(X(t))_{t\geq 0}=(\Xi(\rho^N,\C{N},\C{W},\C{I},\C{J},t))_{t\geq 0}$ be a $(\psi,r,N)$-BBM process constructed as in Section \ref{constructionSection}. We will couple an upper-bounding $N$-BBM process $(X^+(t))_{t\geq 0}$ to $(X(t))_{t\geq 0}$ as follows. Initially, $X^+(0)=\rho^N=X(0)$. As in the construction of the $(\psi,r,N)$-BBM, let $t_0=0$, and let $t_1,t_2,\ldots$ be the times of discontinuity of the Poisson process $\C{N}$ so that at time $t_k$, the $J_k$\textsuperscript{th} leftmost particle of $X^N$ moves to the location of the $I_k$\textsuperscript{th} leftmost. Suppose for induction that $X^+$ is constructed up to time $t_{m-1}$ and that $X(t)\oop X^+(t)$ for $t\leq t_{m-1}$. Then on $(t_{m-1},t_m)$, drive the particle of $X^+(t_{m-1})$ which is at location $\Theta^N_i(X^+(t_{m-1}))$ by the Brownian motion $(W_i(t))_{t\geq t_{m-1}}$. Since $W_i$ is also the Brownian motion which drives the particle of $X$ at location $\Theta^N_i(X(t_{m-1}))$, clearly $X(t_{m-1})\oop X^+(t_{m-1})$ implies that $X(t)\oop X^+(t)$ for $t\in (t_{m-1},t_m)$. Now at the stopping time $t_m$, we branch the $I_m$\textsuperscript{th} leftmost particle of $X^+(t_m-)$ and simultaneously delete the leftmost particle of $X^+(t_m-)$. 

    For a vector $\underline{v}=(v_1,\ldots,v_N)\in \B{R}^N$ and $k,\ell\in [N]$, define the function $L:\B{R}^N\times \B{N}\times \B{N}\to\B{R}^N$ to be the function which deletes the $k$\textsuperscript{th} element and duplicates the $\ell$\textsuperscript{th} element of the vector $\underline{v}$. That is, $$L(\underline{v},k,\ell)=(v_1,\ldots,v_{k-1},v_{k+1},\ldots,v_{\ell-1},v_{\ell},v_{\ell},v_{\ell+1},\ldots,v_N),$$
    where $k$ is not necessarily smaller than $\ell$. With this definition, we have that $X(t_m)=L(X(t_m-),J_m,I_m)$ and we define $X^+(t_m)=L(X(t_m-),1,I_m)$. Then observe that $\underline{u}\oop \underline{v}$ and $j'\leq j$ implies that $L(\underline{u},j,i)\oop L(\underline{v},j',i)$. In layman's terms, \textit{deleting a particle which sits more to the left succeeds in making the whole system sit more to the right}. Therefore $X(t_m)\oop X^+(t_m)$, and hence by induction $X(t)\oop X^+(t)$ for all $t\geq 0$. Thus:
    $$\lim_{t\to\infty}\frac{\Theta^N_1(X(t))}{t} \leq \lim_{t\to\infty}\frac{\Theta^N_1(X^+(t))}{t}=\sqrt{2}-\frac{\pi^2}{\sqrt{2}(\log N)^2} + o\Big(\frac{1}{(\log N)^2}\Big).$$

    Next we construct a lower bound. Consider constructing a $N$-BBM with $\lfloor pN\rfloor$ particles, $(X^-(t))_{t\geq 0}$, as follows. 
    Let $X^-_i(t)$ denote the location of the $i$\textsuperscript{th} leftmost particle of $X^-(t)$, and initially define $X^-_{\lfloor pN\rfloor - i}(0)=X_{N-i}(0)$ for $i \in {\{1,2,\ldots,\lfloor pN\rfloor\}}$. So the $\lfloor pN\rfloor$ particles of $X^-(0)$ have the same locations as the $N$ rightmost particles of $X(0)=\rho^N$. Let $R(X(t),\lfloor pN\rfloor)$ denote be the vector of positions of the rightmost $\lfloor pN\rfloor$ particles of $X(t)$. Similarly to before, assume for induction that $X^-$ is constructed up to time $t_{m-1}$ and that $X^-(t)\oop R(X(t),\lfloor pN\rfloor)$ for $t\leq t_{m-1}$. Then construct $X^-(t)$ for $t\in (t_{m-1},t_m]$ as follows. On the interval $(t_{m-1},t_m)$, drive the particle of $X^-(t_{m-1})$ which is at location $\Theta^{\lfloor pN\rfloor}_{\lfloor pN\rfloor - i}(X^-(t_{m-1})$ at time $t_{m-1}$ by the Brownian motion $W_{N-i}$; that is, we couple the $\lfloor pN\rfloor$ particles of $X^-$ to the rightmost $\lfloor pN\rfloor$ particles of $X$. Thus $X^-(t_{m-1})\oop R(X(t_{m-1}),\lfloor pN\rfloor)$ implies that $X^-(t)\oop R(X(t),\lfloor pN\rfloor)$ for $t\in (t_{m-1},t_m)$. Then at the time $t_m$, the $J_m$\textsuperscript{th} leftmost particle of $X(t_m-)$ moves to the location of the $I_m$\textsuperscript{th} leftmost particle. If $I_m\geq N-\lfloor pN\rfloor + 1$, then we branch the $(I_m-N+\lfloor pN\rfloor)$\textsuperscript{th} leftmost particle of $X^-(t_{m-1}-)$ and simultaneously delete the leftmost particle in the system; otherwise, we do nothing. By assumption \textit{(A2)}, $\psi(x)=0$ on $[1-p,1]$, so $J_m<N-\lfloor pN\rfloor$, which is to say that we never delete one of the $\lfloor pN\rfloor$ rightmost particles of $X$ at a branching/deletion time. 
    
    Clearly if we do nothing, then $X^-(t_m)=X^-(t_{m}-)\oop R(X(t_{m}-),\lfloor pN\rfloor)= R(X(t_{m}),\lfloor pN\rfloor)$. Consider on the other hand the case that a particle of $X^-$ branches at time $t_m$. Again, by assumption \textit{(A2)}, $\psi(x)=0$ on $[1-p,1]$, thus the particle of $X(t_{m})$ which is deleted is not among the $\lfloor pN\rfloor$ rightmost. Therefore $R(X(t_m),\lfloor pN\rfloor )$ consists of the $\lfloor pN\rfloor-1$ rightmost particles of $X(t_m-)$, with the $(I_m-N+\lfloor pN\rfloor)$\textsuperscript{th} of those duplicated. Since $X^-(t_m)$ similarly consists of the $\lfloor pN\rfloor-1$ rightmost particles of $X^-(t_m-)$, with the $(I_m-N+\lfloor pN\rfloor)$\textsuperscript{th} duplicated, we can conclude that $X^-(t_m-)\oop R(X(t_m-),\lfloor pN\rfloor )$ implies $X^-(t_m)\oop R(X(t_m),\lfloor pN\rfloor )$. Therefore by induction it holds that $X^-(t)\oop R(X(t),\lfloor pN\rfloor)$ for all $t\geq 0$. 

    Now observe that, by our construction, the process $(X^-(t))_{t\geq 0}$ is simply an $N$-BBM process with $\lfloor pN\rfloor$ particles, therefore
    \begin{align*}
       \lim_{t\to\infty}\frac{\Theta^N_N(X(t))}{t}\geq \lim_{t\to\infty}\frac{\Theta^N_N(X^-(t))}{t} &= \sqrt{2}-\frac{\pi^2}{\sqrt{2}(\log \lfloor pN\rfloor)^2}+o\Big(\frac{1}{(\log \lfloor pN\rfloor)^2}\Big) \\
        &= \sqrt{2}-\frac{\pi^2}{\sqrt{2}(\log N)^2}+o\left(\frac{1}{(\log N)^2}\right).
    \end{align*}
    It remains to show that the limits $\lim_{t\to\infty}X_N(t)/t$ and $\lim_{t\to\infty}X_1(t)/t$ are equal. To do this, we consider the following events which we can think of as `regeneration events'. For $k\in \B{N}$, define
    $$E^1_k:=\{\C{N}(k+1)-\C{N}(k)=N-1\text{ and }I_{\C{N}(k)+1}=I_{\C{N}(k)+2}=\cdots=I_{\C{N}(k+1)}=N\}.$$
    This is the event that in the time interval $[k,k+1]$, there are exactly $N-1$ branching times, $\tau_1,\tau_2,\ldots,\tau_{N-1}$, at all of which the rightmost particle branches. We can explicitly calculate $\B{P}(E_k^1)=e^{-N}/(N-1)!$. Define $\tau_0=k$ and $\tau_N=k+1$, so that $k=:\tau_0<\tau_1<\cdots<\tau_{N-1}<\tau_N=k+1$. Recall from our construction of the $(\psi,r,N)$-BBM that, in between branching events, the particle which was the $j$\textsuperscript{th} leftmost at the most recent branching event is driven by the Brownian motion $W_j$. Then define
    \begin{align*}E^2_k:=\{W_j(\tau_{i})&\in [W_j(\tau_{i-1})-1/N,W_j(\tau_{i-1})]\text{ for }i\leq N, j\leq N-1\}\\&\cap\{W_N(\tau_i)\in[\max\{W_N(\tau_0)+1,W_N(\tau_{i-1})\},W_N(\tau_0)+2]\text{ for }i\leq N\}\end{align*}
    Note this event is defined so that $\B{P}(E^2_k|E^1_k)=c(N)>0$ for some $c(N)$ independently of $k$ and the configuration of $X(k)$. Finally, let $$E^3_k:=\{J_{\C{N}(k)+1}=J_{\C{N}(k)+2}=\ldots=1\}$$ be the event that all branching events in $[k,k+1]$, the leftmost particle is deleted. Since $\psi(x)>0$ for all $[0,\epsilon]$, this happens with positive probability; we can explicitly calculate $\B{P}(E^3_k|E^1_k)=(\int_0^{1/N}\psi(u)du)^{N-1}$. 
    
    Now define $E_k:=E_k^1\cap E_k^2\cap E_k^3$. Therefore $\B{P}(E^1_k\cap E^2_k \cap E_k^3)=d(N)>0$ for some $d(N)$ independently of $k$ and the configuration of $X(k)$. Now let us consider the configuration of particles of $X$ at time $k+1$ on the event $E_k$. Since $W_j(\tau_i)\leq W_j(\tau_{i-1})$ for all $j=1,2,\ldots,N-1$, this means that, at time $\tau_\ell$, every particle which has not yet been killed has a location left of $\Theta^N_N(X^N(k))$. Moreover, since $W_j(\tau_i)\geq W_j(\tau_{i-1})-1/N$ and $W_N(\tau_i)\geq W_N(\tau_0)+1$, this means that every particle which has been born during $[k,k+1]$ has a location to the right of $\Theta^N_N(X^N(k))$. Therefore, on the events $E_k^1$, $E_k^2$, and $E_k^3$, the $N-1$ branching times which occur on $[k,k+1]$ cause every particle to become an offspring of the rightmost particle and have location in $[\Theta^N_N(X^N(k)),\Theta^N_N(X^N(k))+2]$ by time $k+1$.
    
    So define $T_0=0$ and subsequently $T_i:=\inf\{k>T_{i-1}:k\in \B{N}\text{ and }E_k\text{ occurs}\}$. Since $E_k$ occurs for each $k$ independently, with positive probability, and independently of the initial configuration, thus $T_{i+1}-T_{i}$ are i.i.d. with $\B{E}[T_{i+1}-T_{i}]=1/d(N)<\infty$ for $i=1,2,\ldots$. Moreover, $T_i\geq i$, so $T_i\to\infty$ as $i\to\infty$ and $|\Theta^N_N(X(T_i))-\Theta^N_1(X(T_i))|\leq 2$ for each $i=1,2,\ldots$. Therefore:
    $$\lim_{t\to\infty}\left|\frac{\Theta^N_N(X(t))}{t}-\frac{\Theta^N_1(X(t))}{t}\right|=\lim_{i\to\infty}\left|\frac{\Theta^N_N(X(T_i))-\Theta^N_1(X(T_i))}{T_i}\right|=0.$$

    Therefore we can conclude that there is an asymptotic velocity $v_N:=\lim_{t\to\infty}X_1(t)/t=\lim_{t\to\infty}X_N(t)/t$. Moreover, by sandwiching:
    $$\sqrt{2}-\frac{\pi^2}{\sqrt{2}(\log N)^2}+o\left(\frac{1}{(\log N)^2}\right)\leq \lim_{t\to\infty}\frac{\Theta^N_N(X(t))}{t}=v_N^\psi=\lim_{t\to\infty}\frac{\Theta^N_1(X(t))}{t}\leq \sqrt{2}-\frac{\pi^2}{\sqrt{2}(\log N)^2}+o\left(\frac{1}{(\log N)^2}\right),$$
    thus proving the claim.
\end{proof}

Within the assumption \textit{(A2)}, the requirement that $\psi(x)=0$ on some left neighbourhood of 1 ensures that the process still propagates with speed close to $\sqrt{2}$. Intuitively, if we were to kill the rightmost particle at a strictly positive rate, we would no longer expect the front to move with speed close to $\sqrt{2}$, as it does in branching Brownian motion or in the $N$-BBM. In our argument, this hypothesis is used precisely when comparing the system with an $N$-BBM. %We in fact expect that, if $\psi(x)\to 0$ sufficiently quickly as $x\to 1$ then the asymptotic velocity should still be close to $\sqrt{2}$, but our coupling method does not apply in that regime. 
Conversely, we need $\psi$ to be strictly positive on some neighbourhood of $0$ to ensure that the entire particle cloud shares the same asymptotic velocity. If we never killed the leftmost particle, we would expect it to drift to $-\infty$ while the rightmost particle drifts to $\infty$, so that no asymptotic speed exists. This latter condition is used to show that $\lim_{t\to\infty}\Theta^N_1(X(t))/t$ and $\lim_{t\to\infty}\Theta^N_N(X(t))/t$ coincide.

Although it is not possible to prove with the coupling method above, we actually believe that this asymptotic speed result still holds under weaker conditions on $\psi$:

\vspace{1em} 

\begin{conj}\label{speedConjecture} Let $\psi:[0,1]\to[0,\infty)$ be such that $\psi(x)>\psi$ for $x\in [0,\epsilon]$ for some $\epsilon>0$ and $\psi(1-h)=o(h)$ as $h\to 0$. Then the $(\psi,1,N)$-BBM has asymptotic velocity $v_N^\psi=\sqrt{2}-\frac{\pi^2}{\sqrt{2}(\log N)^2}+o\left(\frac{1}{(\log N)^2}\right)$ as $N\to\infty$. 
\end{conj}

It remains to study the travelling wave speed of the corresponding PDE $U_t = \frac12 U_{xx} + U - \int_{1-U}^1 \psi(s)ds$ under assumptions \textit{(A1)} and \textit{(A2)}.

It is a well known result (see \cite{hamelNadin}, for example) that if the function $G$ satisfies the conditions $G(0)=G(1)=0$, $G(x)>0$ for all $x\in (0,1)$, and $G'(x)\leq G'(0)$ for all $x\in [0,1]$, then $U_t = \frac12 U_{xx}+G(U)$ has a travelling wave solution $U(x,t)=\omega_c(x-ct)$ with $\lim_{z\to-\infty}\omega_c(z)=1$ and $\lim_{z\to\infty}\omega_c(z)=0$ for all speeds $c\geq \sqrt{2G'(0)}$. Therefore we can state the following:

\vspace{1em}

\begin{propn}
    Let $\psi$ be such that $\int_{1-x}^1 \psi(s)ds < x$ for all $x\in [0,1]$ and $\psi(x)=0$ for $x\in [1-p,1]$ for some $p\in (0,1)$. Then the PDE (\ref{mainResultPDE}) has a travelling wave solution for all speeds $c\geq \sqrt{2}$. 
\end{propn}

\begin{proof}
    Defining $G(x):=x-\int_{1-x}^1 \psi(s)ds$, it immediately follows that $G(0)=0$, and $\int_0^1 \psi(s)ds = 1$ gives that $G(1)=0$. Moreover, since $\psi$ is positive and $\psi(x)=0$ for $x\in [1-p,1]$, $G'(x)=1-\psi(1-x)\leq 1=G'(0)$ for all $x\in [0,1]$. Finally, the condition $\int_{1-x}^1 \psi(s)ds < x$ ensures that $G(x)>0$ for $x\in (0,1)$. 
\end{proof}

The condition that $\psi$ is $0$ near $1$ is equivalent to $G'$ being maximal $1$ at $0$ and hence that the minimal travelling wave speed is $\sqrt{2G'(0)}=\sqrt{2}$. Intuitively, if $\psi$ were bounded below by $\delta\in (0,1)$ on $[0,1]$, with $\psi(x)=\delta$ for $x\in [1-p,1]$, the PDE would have minimal travelling wave speed $\sqrt{2(1-\delta)}$. In terms of the $N$-particle system, if we are killing the rightmost particle at a positive rate $N\delta$, this implies that the cloud of particles would then have the slower speed $\sqrt{2(1-\delta)}<\sqrt{2}$. We therefore conjecture the following:

\vspace{1em}

\begin{conj}
    Let $\psi$ be a selection function such that $\psi(x)\geq \delta$ for all $x\in [0,1]$ and $\psi(x)=\delta$ for all $x\in [1-p,1]$ for some $p \in (0,1)$. Then the $(\psi,1,N)$-BBM has asymptotic velocity:
    $$v^\psi_N=\sqrt{2(1-\delta)}-\frac{c}{(\log N)^2} + o\left(\frac{1}{(\log N)^2}\right)$$
    for some positive constant $c$.
\end{conj}

On the other hand, if $G(x)$ is negative on $(0,1)$, we may not have a travelling wave at all. We recall the following result:

\vspace{1em}

\begin{theorem*}[Ch. 1, \S 3, Theorem 3.14, \cite{vvv}]
    Consider the PDE $U_t = U_{xx} + G(U)$, with $G(0)=G(1)=0$ and $G'(0),G'(1)>0$. Then there is no travelling wave solution $U(x,t)=w(x-ct)$ of the PDE with $\lim_{x\to-\infty}w(x)=1$ and $\lim_{x\to\infty}w(x)=0$. 
\end{theorem*}

Therefore we can conclude that \textbf{no weak selection principle holds in general}, since there exist selection functions $\psi$ such that the $(\psi,1,N)$-BBM has asymptotic speed converging to $\sqrt{2}$ but the PDE (\ref{mainResultPDE}) does not have travelling wave solutions. 

We may have in mind the case when $\epsilon \leq \psi(s) \leq \delta < 1$ for $s\in [0,\epsilon]$ and $\psi(x)=0$ for $x\in [1-p,1]$ so that $\psi$ satisfies assumption \textit{(A2)}, and thus by Theorem \ref{asympSpeedTheorem} the $(\psi,1,N)$-BBM has an asymptotic velocity, but the PDE has no travelling wave solutions. The intuition behind this disparity is as follows. If $\psi(s)<1$ for $s\in [0,\epsilon]$, then in the $N$-particle system, the leftmost $\lfloor \epsilon N\rfloor$ particles branch more frequently than they are killed; rate $\epsilon$ versus rate $\delta \epsilon$. When $N$ is finite, this is controlled by the fact that we have rare events in which, for example, all particles come within distance $1$ of one another. However the rarity of these events is such that, as $N\to\infty$, this is in some sense `insufficient to stop some of the mass diffusing to $-\infty$ in the PDE'.

Whilst no weak selection principle holds in general, we may still prove that the solution of the PDE spreads to the right with speed at least $\sqrt{2}$ in a sense which we will now define. For more complex equations, we may see more complex propagating fronts, and there exists `terrace solutions' in which several travelling fronts of different speeds may be `stacked on top of each other' (see \cite{ducrotTerrace}, \cite{gilettiTerraces}, for example).

\vspace{1em}

\begin{defn}
    Consider the PDE $U_t = \frac12 U_{xx} + G(U)$ with initial condition $U(x,0)=U_0(x)$ such that $U_0(x)\to 1$ as $x\to-\infty$ and $U_0(x)\to 0$ as $x\to\infty$. We say that the solution $U(x,t)$ \textit{spreads up to (the value) $u^\star$ at speed $c^\star$} if:
    \begin{enumerate}
        \item $\lim_{t\to\infty}\inf_{x\leq ct} U(x,t)\geq u^\star$ for all $c<c^\star$,
        \item $\lim_{t\to\infty}\sup_{x\geq ct} U(x,t)=0$ for all $c>c^\star$.
    \end{enumerate}
    More generally, we may say that the solution \textit{spreads up to (the value) $u^\star$ with speed at least $c^\star$} if just condition 1. holds.
\end{defn}

This definition is a slight generalisation of the definition given by Hamel and Nadin in \cite{hamelNadin}, allowing for the possibility that the solution doesn't spread up to the value $1$, but may spread up to a strictly smaller value $u^*\in (0,1)$. Let us observe also the somewhat trivial fact that, as $U$ is monotonic decreasing, if $\lim_{t\to\infty}\inf_{x\leq ct}U(x,t)\geq a$, then also $\lim_{t\to\infty}\inf_{x\leq c't}U(x,t)\geq a$ for all $c'<c$. 

Following \cite{hamelNadin}, we also introduce the following definition of a \textit{front-like} initial condition $U_0(x)$:

\vspace{1em}

\begin{defn}[\cite{hamelNadin}, Definition 1.1]
    We say that a function $U_0\in L^\infty(\B{R})$ is \textit{front-like} if $0\leq U_0(x)\leq 1$ for almost all $x\in \B{R}$, $\lim_{x\to\infty}\sup_{y\geq x}|U_0(x)|=0$, and there exists $x_-\in \B{R}$ and $\delta>0$ such that $U_0(x)\geq \delta$ for almost all $x<x_-$.
\end{defn}

\vspace{1em}

\begin{propn} Let $G:[0,1]\to \B{R}$ be a differentiable function such that $G(0)=G(u^\star)=G(1)=0$, $G(x)>0$ for $x\in (0,u^\star)$ and $G'(x)\leq G'(0)$ for $x\in [0,1]$, and consider the PDE 
\begin{equation}\label{PDEtoCompare} U_t = \frac12 U_{xx} + G(U).\end{equation}
Then any front-like initial condition $U_0(x)$ such that $\lim_{x\to-\infty}U_0(x)=1$ and $\lim_{x\to\infty}U_0(x)=0$ spreads up to the value $u^\star$ with speed at least $\sqrt{2G'(0)}$. Moreover, the minimal spreading speed is attained; that is, there exists an initial condition $U_0(x)$ which spreads up to the value $u^\star$ at speed $\sqrt{2G'(0)}$.%for each $c\geq \sqrt{2G'(0)}$, there exists an initial condition $U^c_0(x)$ such that $\lim_{x\to-\infty}U^c_0(x)=1$ and $\lim_{x\to\infty}U^c_0(x)=0$ such that the solution $U(x,t)$ to the PDE with initial condition $U^c_0(x)$ propagates up to $u^\star$ with speed $c$, and for $c<\sqrt{2G'(0)}$ there is no initial condition which propagates up to $u^\star$ with speed $c$. 
\end{propn}

\begin{proof}
    We prove the above proposition by using comparison principles to bound the solution of the PDE (\ref{PDEtoCompare}) above and below. First we construct a lower bound to show that every initial condition $U_0$ spreads up to $u^\star$ with speed at least $\sqrt{2G'(0)}$. 
    We prove the above proposition by using comparison principles to bound the solution of the PDE (\ref{PDEtoCompare}) above and below. First we construct a lower bound to show that every initial condition $U_0$ spreads up to $u^\star$ with speed at least $\sqrt{2G'(0)}$. 
    
    To construct the lower bound, recall that solutions to the F-KPP equation are ordered with respect to their initial condition. That is, if $U_0(x) \leq \tilde{U}_0(x)$ for all $x\in \B{R}$, and $U(x,t)$ and $\tilde{U}(x,t)$ solve the PDE (\ref{PDEtoCompare}) with initial conditions $U_0(x)$ and $\tilde{U}_0(x)$ respectively, then $U(x,t)\leq \tilde{U}(x,t)$ for all $(x,t)\in \B{R}\times [0,\infty)$ (see for example \cite{kpp}, Theorem 2). Now fix a front-like initial condition $U_0$ with $\lim_{x\to-\infty}U_0(x)=1$ and $\lim_{x\to\infty}U_0(x)=0$, and define $U^-_0(x):=\min\{U_0(x),u^\star\}$, and let $U^-(x,t)$ be the solution to the PDE (\ref{PDEtoCompare}) with initial condition $U^-_0$. We can consider this as a scaled F-KPP with source term $G$ such that $G(0)=G(u^\star)=0$ and $G(x)>0$ for $x\in (0,u^\star)$ and $G'(x)$ maximal at $0$. %Therefore $U^-$ has travelling wave solutions $U^-(x,t)=\omega_c(x-ct)$ with $\lim_{z\to-\infty}\omega_c(z)=u^\star$ and $\lim_{z\to\infty}\omega_c(z)=0$ for every speed $c\geq \sqrt{2G'(0)}$. 
    Moreover, the initial condition $U^-_0$ is still front-like. Therefore the inequality (1.6) of \cite{hamelNadin} gives us that:
    $$\sqrt{2G'(0)} \leq \sup\left\{c\in \B{R}: \lim_{t\to\infty}\inf_{x\leq ct}U^-(x,t)=u^\star\right\},$$
    so that $U^-$ spreads up to $u^\star$ with speed at least $\sqrt{2G'(0)}$. Then by comparison:
    $$\lim_{t\to\infty}\inf_{x\leq ct}U(x,t)\geq \lim_{t\to\infty}\inf_{x\leq ct}U^-(x,t)\geq u^\star$$ for all $c<\sqrt{2G'(0)}$, therefore $U(x,t)$ spreads up to $u^\star$ with speed at least $\sqrt{2G'(0)}$. 
    
    Next, by bounding the solution above, we show that there exists an initial condition which spreads up to $u^\star$ at speed $\sqrt{2G'(0)}$. We do this by recalling that solutions to the F-KPP are ordered with respect to the source term $G$. That is, if $G(x)\leq H(x)$ for all $x\in \B{R}$, $U$ solves $U_t = \frac12 U_{xx}+ G(U)$ with initial condition $U_0$, and $V$ solves $V_t = \frac12 V_{xx} + H(V)$ with initial condition $U_0$, then $U(x,t)\leq V(x,t)$ for all $(x,t)\in \B{R}\times [0,\infty)$ (see for example \cite{kpp}, Theorem 3). 
    
    So fix $\epsilon \in (0,u^\star)$ and let $G^+$ be such that $G^+(x)=G(x)$ for $x\in [0,\epsilon]$, $G^+(x)>0$ and $G^+(x)\geq G(x)
    $ for $x\in [\epsilon, 1]$, and $\frac{d}{dx}G^+(x)\leq \frac{d}{dx}G^+(0)$ for $x\in [0,1]$. Therefore the PDE $U_t = \frac12 U_{xx} + G^+(U)$ has, for every speed $c\geq \sqrt{2G'(0)}$, a travelling wave solution $U^+(x,t)=U^+_c(x-ct)$ such that $\lim_{z\to-\infty}U^+_c(z)=1$ and $\lim_{z\to\infty}U^+_c(z)=0$. Fix $c^\star = \sqrt{2G'(0)}$, and let $U(x,t)$ be the solution to (\ref{PDEtoCompare}) with initial condition $U^+(x,0)=U^+_{c^\star}(x)$. Then by comparison, for any $c>c^\star=\sqrt{2G'(0)}$:
    $$\lim_{t\to\infty}\sup_{x\geq ct}U(x,t) \leq \lim_{t\to\infty}\sup_{x\geq ct}U^+_{c^\star}(x-c^\star t)=0.$$
    Combining this with the previous result confirms that the solution $U(x,t)$ to (\ref{PDEtoCompare}) with initial condition $U^+_{c^\star}(x)$ spreads up to $u^\star$ at speed $\sqrt{2G'(0)}$, thus concluding the proof of the proposition. \end{proof}

    This therefore gives a weaker sense in which the weak selection principle holds under the assumptions \textit{(A1)} and \textit{(A2)}.
    
\appendix
\section{Hydrodynamic limit of branching Brownian motion}
In this appendix, we provide a brief proof of a hydrodynamic limit result for branching Brownian motion for completeness. Although the author believes that the result is well-known in the field, a precise statement and proof of the specific result cannot be found. A somewhat similar result is proven in Appendix A of \cite{beckman}. Many of the techniques are identical to those used in the main body of this paper.

\vspace{1em}

\begin{theorem} \label{bbmHydroAppendix}
    Let $\rho^N$ be a sequence of probability measures such that $\rho^N$ is the sum of $N$ atoms of weight $1/N$, and such that $\rho^N\Rightarrow \rho\ll\lambda$ as $N\to\infty$. Consider a branching Brownian motion $\{B^N_i(t):i\in \C{U}^N(t)\}_{t\geq 0}$ which initially has $N$ particles and empricial distribution $\rho^N$, and in which each particle branches at time-dependent rate $r(t)$. Define
    $$\eta_t^N:=\frac1N\sum_{i\in \C{U}^N(t)}\delta_{B^N_i(t)},$$
    so that $\eta_0^N=\rho^N$, and $(\eta_t^N)_{t\geq 0}$ is the empirical measure-valued process describing the system. Let $\C{L}^N_T$ be the law of $(\eta_t^N)_{t\in [0,T]}$. Then $\C{L}^N_T$ converges weakly to $\C{L}^\infty_T$ in the Skorokhod topology, and if $(\eta_t)_{t\in [0,T]}\sim \C{L}^\infty_T$, then $\eta_t(dx)=u(x,t)dx$ where $u$ is the unique classical solution to the PDE $u_t = \frac12 u_{xx}+ r(t)u$ for $(x,t)\in \B{R}\times (0,T)$ and $u(x,0)=\rho(x)$. 
\end{theorem}

\begin{proof}
    Note that the BBM process with branching rate $r$ is just the $(\psi,r,N)$-BBM without deletion of particles. Therefore reasoning exactly as in Proposition \ref{weakRepPropn}, we can write
    \begin{align} \label{bbmRep}\langle \eta_t^N,f(\cdot,t)\rangle = \langle \rho^N,f(\cdot,t)\rangle + \int_0^t \langle \eta_s^N,\frac12 f_{xx}(\cdot,s)+f_t(\cdot,s)+r(s)f(\cdot,s)\rangle ds + \tilde{M}_t^{N,W}+\tilde{M}^{N,P}_t    \end{align}
    where $\tilde{M}^{N,W}$ and $\tilde{M}^{N,P}$ are local martingales with $\B{E}[[\tilde{M}^{N,W}]_t]$ and $\B{E}[[\tilde{M}^{N,P}]_t]$ converging to $0$ as $N\to\infty$. 

    Next we prove tightness. We will prove tightness of the sequence $((\eta_t^N)_{t\in [0,T]})_{N=1,2,\ldots}$ in the space $\C{D}([0,T],\C{M}^v_F(\B{R}))$, where $\C{M}_F^v(\B{R})$ is the space of finite measures on $\B{R}$ with the vague topology. It is shown by Roelly-Coppoletta \cite{roellyCop} that to show tightness in this space it is sufficient to show tightness of $(\langle \eta_t^N,f\rangle)_{t\in [0,T]}$ for all $f$ in a dense subset of $C_0(\B{R})$, the space of continuous functions vanishing at infinity. We will take for our dense subset of $C_0(\B{R})$ the set $\C{BC}^{2,1}(\B{R}\times [0,\infty),\B{R})\cap C_0(\B{R})$. As in the proof of Proposition \ref{tightness}, by Aldous' criterion (Theorem 16.10, \cite{billingsley}), to prove tightness of $(\langle \eta_t^N,f\rangle)_{t\in [0,T]}$, it is sufficient to show two conditions A and B. 

    Now the expected number of particles in the BBM with branching rate $r$ and with initially $N$ particles is $N\exp(\int_0^t r(s)ds)$. Therefore by the Markov inequality:
    \begin{align*}
        \B{P}(\sup_{0\leq t\leq m}|\langle \eta_t^N,f\rangle|\geq a)\leq \frac1a\B{E}[\sup_{0\leq t\leq m}|\langle \eta_t^N,f\rangle|]\leq \frac{\|f\|_\infty}{a}\exp(\int_0^m r(s)ds),
    \end{align*}
    so that for every $m$, $\lim_{a\to\infty}\limsup_{N\to\infty}\B{P}(\sup_{0\leq t\leq m}|\langle\eta_t^N,f\rangle|\geq a)=0$. This proves condition A. For condition B, take $\epsilon,\eta,m>0$ and stopping time $\tau\leq m$. Then by \eqref{bbmRep}, since $f$, $f_{xx}$ and $r$ are all bounded, say by $C$, we have:
    \begin{align*}|\langle \eta^N_{\tau+\delta},f\rangle - \langle \eta^N_\tau,f\rangle|\leq \int_{\tau}^{\tau+\delta}(C+C^2)\langle \eta_s^N,1\rangle ds + |\tilde{M}_{\tau+\delta}^{N,W}-\tilde{M}_{\tau}|+|\tilde{M}_{\tau+\delta}^{N,P}-\tilde{M}_{\tau}^{N,P}|.\end{align*}
    As before, we noted that $\langle \eta_t^N,1\rangle$ has expectation $\exp(\int_0^t r(s)ds)$, therefore by the Markov inequality an the Burkholder-Davis-Gundy inequality, there exists $\kappa$ such that:
    \begin{align*}
        \B{P}(|\langle \eta_{\tau+\delta}^N,f\rangle &- \langle \eta_\tau^N,f\rangle|\geq \epsilon)\\
        &\leq \B{P}(C(1+C)\delta\langle \eta_{m+\delta}^N,1\rangle\geq \epsilon/3) + \B{P}(|\tilde{M}_{\tau+\delta}^{N,W}-\tilde{M}_{\tau}|^2>\epsilon^2/9)+\B{P}(|\tilde{M}_{\tau+\delta}^{N,P}-\tilde{M}_{\tau}^{N,P}|^2>\epsilon^2/9)\\
        &\leq \frac{1}{\epsilon^2}\left(3C(1+C)\epsilon\delta\B{E}[\langle \eta_{m+\delta}^N,1\rangle]+9\B{E}[|\tilde{M}_{\tau+\delta}^{N,W}-\tilde{M}_{\tau}|^2]+9\B{E}[|\tilde{M}_{\tau+\delta}^{N,P}-\tilde{M}_{\tau}^{N,P}|^2]\right)\\
        &\leq \frac{1}{\epsilon^2}\left(3C(1+C)\epsilon \delta e^{(m+\delta)\|r\|_\infty}+9\kappa\B{E}[[\tilde{M}^{N,W}]_\delta]+9\kappa\B{E}[[\tilde{M}^{N,P}]_\delta]\right)\xrightarrow[\delta\to 0,N\to\infty]{}0,
    \end{align*}
    which proves condition B. Therefore we can conclude that the sequence $((\eta_t^N)_{t\in [0,T]})_{N=1,2,\ldots}$ is tight. Therefore $(\C{L}^N_T)_{N=1,2,\ldots}$ has a subsequential limit $\C{L}^\infty_T$ to which it converges weakly. Since $N\langle \eta^N_T,1\rangle$ has the law of the sum of $N$ independent geometric random variables with parameter independent of $N$, thus $\B{E}[\langle \eta_T^N,1\rangle^2]$ has an upper bound independent of $N$. Thus for $0<s<t<T$, we can bound
    \begin{align*}\B{E}[|\langle \eta^N_s,f\rangle-\langle \mu_t^N,f\rangle|^2]\leq 9(C^2(1+C)^2(t-s)^2\B{E}[\langle \eta_T^N\rangle^2]+\B{E}[|\tilde{M}^{N,W}_t-\tilde{M}^{N,W}_s|^2]+\B{E}[|\tilde{M}^{N,P}_t-\tilde{M}^{N,P}_s|^2]),\end{align*}
    and hence, as in the proof of Proposition \ref{asContinuity}, by Fatou's lemma, there exists a constant $\tilde{C}$ such that $$\B{E}[|\langle \eta_s^\infty,f\rangle - \langle \eta_t^\infty\rangle|^2]\leq \tilde{C}|s-t|^2$$ for all $0<s<t<T$. Therefore by the Kolmogorov continuity theorem $\langle \eta_t^\infty,f\rangle$ has a continuous modification and hence, as it is also cadlag, $t\mapsto \langle \eta_t^\infty,f\rangle$ is almost surely continuous (see for example Theorem 1 of \cite{schilling}). This holds for all $f\in \C{BC}^{2,1}(\B{R}\times [0,\infty),\B{R})\cap C_0(\B{R})$, and hence $t\mapsto \eta_t^\infty$ is almost surely continuous with respect to the vague topology. Now by the Skorokhod representation theorem, there exist $(\tilde{\eta}^N_t)_{t\in [0,T]}$ and $(\tilde{\eta}^\infty_t)_{t\in [0,T]}$ defined on the same probability space $(\Omega,\C{F},\B{P})$ such that $\tilde{\eta}^N$ converges to $\tilde{\eta}^\infty$ in the Skorokhod topology $\B{P}$-a.s., with $(\tilde{\eta}^N_t)_{t\in [0,T]}\sim \C{L}^N_T$ and $(\tilde{\eta}^\infty_t)_{t\in [0,T]}\sim \C{L}^\infty_T$. Since $(\tilde{\eta}^N_t)_{t\in [0,T]}$ converges almost surely to $(\tilde{\eta}^\infty_t)_{t\in [0,T]}$ in the Skorokhod topology and $t\mapsto \tilde{\eta^\infty_t}$ is almost surely continuous on $[0,T]$, thus for any $f(x,t)\in \C{BC}_0^{2,1}(\B{R}\times [0,\infty),\B{R})$, $\langle \tilde{\eta}^N_t,f(\cdot,t)\rangle$ converges $\langle \tilde{\eta}^\infty_t,f(\cdot,t)\rangle$ for all $t\in [0,T]$ almost surely. Hence by the dominated convergence theorem:
    \begin{align*}
        \tilde{G}(\tilde{\eta}^N,f,t)&:=\langle \tilde{\eta}_t^N,f(\cdot,t)\rangle - \langle \rho^N,f(\cdot,t)\rangle - \int_0^t \langle \tilde{\eta}_s^N,\frac12 f_{xx}(\cdot,s)+f_t(\cdot,s)+r(s)f(\cdot,s)\rangle ds \\
        &\xrightarrow[N\to\infty]{\B{P}\text{-a.s.}}\langle \tilde{\eta}_t^\infty,f(\cdot,t)\rangle -\langle \rho,f(\cdot,t)\rangle - \int_0^t \langle \tilde{\eta}^\infty_s,\frac12 f_{xx}(\cdot,s)+f_t(\cdot,s)+r(s)f(\cdot,s)\rangle ds.
    \end{align*}
    Moreover, by \eqref{bbmRep} and the Burkholder-Davis-Gundy inequality, $\B{E}[\tilde{G}(\tilde{\eta}^N,f,t)^2] \leq 4\B{E}[[\tilde{M}^{N,W}]_t]+4\B{E}[[\tilde{M}^{N,P}]_t]$ which converges to $0$ as $N\to\infty$. Now fix $t$ and $f$, and observe that 
    \begin{align*}|\tilde{G}(\tilde{\eta}^N,f,t)|=|\langle \tilde{\eta}_t^N,f(\cdot,t)\rangle - \langle \rho^N,f(\cdot,t)\rangle &- \int_0^t \langle \tilde{\eta}_s^N,\frac12 f_{xx}(\cdot,s)+f_t(\cdot,s)+r(s)f(\cdot,s)\rangle ds|\\
    &\leq C\langle \tilde{\eta}^N_t,1\rangle + C + (2C + C^2)t\langle \tilde{\eta}^N_t,1\rangle=:\Sigma_N.\end{align*}
    Then for any $K>0$:
    \begin{align*}
        \B{E}[\Sigma_N\is_{\{|\Sigma_N|\geq K\}}]\leq \B{E}\left[\frac{\Sigma_N^2}{K}\is_{\{|\Sigma_N|\geq K\}}\right]\leq K^{-1}\B{E}[\Sigma_N^2].
    \end{align*}
    Recalling that $\B{E}[\langle \eta_t^N,1\rangle]$ and $\B{E}[\langle \eta_t^N,1\rangle^2]$ have bounds independent of $N$, thus $\lim_{K\to\infty}\sup_{N}\B{E}[|\Sigma_N|\is_{\{|\Sigma_N|\geq K\}}]=0$, which is to say that $(\Sigma_N)_{N=1,2,\ldots}$ is uniformly integrable. As a result $(|G(\tilde{\eta}^N,f,t)|)_{N=1,2,\ldots}$ is uniformly integrable. Since $(|G(\tilde{\eta}^N,f,t)|)_{N=1,2,\ldots}$ is uniformly integrable and converges to $0$ in expectation, thus it converges to $0$ almost surely. Therefore $\eta_t^\infty$ is almost surely a solution to the equation:
    \begin{align} \label{bbmWeakSoln}
        \langle \eta_t^\infty,f(\cdot,t)\rangle = \langle \rho,f(\cdot,0)\rangle + \int_0^t \langle \eta_s^\infty,\frac12 f_{xx}(\cdot,s)+f_t(\cdot,s)+r(s)f(\cdot,s)\rangle ds
    \end{align}
    for any $f(x,t)\in \C{BC}^{2,1}_0(\B{R}\times [0,\infty),\B{R})$. Since the unique classical solution of the PDE $u_t = \frac12 u_{xx} + r(t)u$ with $u(x,0)=\rho(x)$ certainly solves \eqref{bbmWeakSoln}, it only remains to show that there is at most one solution to \eqref{bbmWeakSoln}.

    Suppose that $\eta_t^\infty$ and $\hat{\eta}_t^\infty$ are two solutions to \eqref{bbmWeakSoln}. Let $\phi:\B{R}\to\B{R}$ be a smooth compactly supported test function, and let $f$ be the solution of the backwards heat equation $f_t=\frac12f_{xx}+rf$ with terminal condition $f(x,t)=\phi(x)$. Then it immediately follows that $\langle \eta_t^\infty, \phi\rangle = \langle \hat{\eta}_t^\infty, \phi\rangle$. Thus $\eta_t^\infty=\hat{\eta}_t^\infty$ for any $t$. Hence $\eta^\infty_t(dx)=u(x,t)dx$ where $u$ is the unique classical solution to the PDE $u_t=\frac12 u_{xx}+r(t)u$ with $u(x,0)dx=\rho(dx)$, thus concluding the proof. 
\end{proof}

\section*{Acknowledgements}
The author thanks Julien Berestycki for his careful and insightful supervision of the project, and Jo\~{a}o Luiz de Oliveira Madeira for his generous help and advice on a number of parts of this work. The author also wishes to express gratitude to Paul de Lambert for fruitful discussions relating to this work and his permission to produce the coupling idea contained in the proof of Theorem \ref{asympSpeedTheorem}, and Matthias Winkel and Christina Goldschmidt for their comments on an earlier version for this work. 

\bibliographystyle{plain}
\bibliography{bps}

\end{document}